\documentclass[12pt]{amsart}
\usepackage{amsmath,amsthm,amssymb,amsfonts}
\usepackage[active]{srcltx}
\usepackage{latexsym}
\usepackage{enumerate}
\usepackage{verbatim}
\usepackage{multicol}
\usepackage{caption}
\usepackage{graphicx}
\usepackage{epsfig}
\usepackage{epstopdf}
\usepackage{color}
\usepackage{bezier}
\usepackage{curves}
\usepackage{fullpage}

\newtheorem{theorem}{Theorem}
\newtheorem{claim}{Claim}

\newtheorem{problem}{Problem}
\newtheorem{problems}[problem]{Problem}

\newtheorem{remark}{Remark}
\newtheorem{lemma}[theorem]{Lemma}

\newtheorem{proposition}[theorem]{Proposition}
\newtheorem{corollary}[theorem]{Corollary}
\newtheorem{example}{Example}

\newtheorem{definition}{Definition}

\def\rem{{\bf Remark. }}

\newcommand{\QQ}{{\mathbb Q}}
\newcommand{\NN}{{\mathbb N}}
\newcommand{\ZZ}{{\mathbb Z}}
\newcommand{\RR}{{\mathbb R}}

\newcommand{\FF}{{\mathbb F}}

\newcommand{\Int}{{Int}}

\title{Interval orders, semiorders and ordered groups}
\author [M.Pouzet]{Maurice Pouzet}
\address{Univ. Lyon, Universit\'e Claude-Bernard Lyon1, CNRS UMR 5208, Institut Camille Jordan,
43, Bd. du 11 Novembre 1918,
69622 Villeurbanne, France et Department of Mathematics and Statistics, The University of Calgary, Calgary, Alberta, Canada} \email{pouzet@univ-lyon1.fr }
\author[I.Zaguia]{Imed Zaguia*}\thanks{*Corresponding author}
\address{Department of Mathematics \& Computer Science, Royal Military College,
P.O.Box 17000, Station Forces, Kingston, Ontario, Canada K7K 7B4}
\email{zaguia@rmc.ca}

\date{\today}

\begin{document}

\begin{abstract} We prove that the order of an ordered group is an interval order if and only if it is a semiorder. Next, we prove that every semiorder is isomorphic to a collection $\mathcal J$ of intervals of some totally ordered abelian group, these intervals being of the form $[x, x+ \alpha[$ for some positive $\alpha$. We describe ordered groups such that the ordering is a semiorder and we introduce threshold groups generalizing totally ordered groups. We show that the free group on finitely many generators and the Thompson group $\mathbb F$ can be equipped with a compatible semiorder which is not a weak order. On another hand, a group introduced by Clifford cannot. \end{abstract}
\maketitle

\input{epsf}

\section{Introduction}

This paper is about the interactions between order and group operation in an ordered group. It focuses on semiorders and interval orders, two notions whose role in Mathematical Psychology is well documented through books and papers,  see  for example \cite{aleskerov-al, bgr, bridges, pirlotvincke} and \cite{bogart, brt,candeal1, candeal,doignon-al1,doignon-al2,fi,luce56,luce87,manders,rabinovitch,scottsuppes,wiener}. Interval orders are much more general than semiorders. But, surprisingly, a compatible order on a group is a semiorder whenever it is an interval order. This result, which is not difficult to prove,  lead us to investigate the structure of ordered groups whose order is a semiorder. Among those, we identify threshold groups (for which the quasi-orders $\leq_{pred}$ and $\leq_{succ}$, naturally associated to any order $\leq$, coincide and are total orders).  A basic example is the additive group $\RR$ of real numbers with the order defined by $x\leq_1 y$ if either $x=y$ or $y-x\geq 1$ for the linear order on the reals. These ordered groups generalize totally ordered groups. We obtain several structural results. We  prove for example that these ordered groups are rich enough to embed every semiorder, thus providing an extension of Scott-Suppes's representation of semiorders. It is natural to ask   which groups can be endowed with a nontrivial semiorder (that is distinct from the equality relation) or a threshold order. We prove that a group  can be equipped with a semiorder distinct from the equality relation (hence with a nontrivial semiorder) if and only if  it  admits  a totally orderable quotient by some proper subgroup. If $G$ is abelian this condition amounts to the fact that the torsion subgroup is distinct from $G$.  We have only partial answers  for the existence of threshold order. While it is true that
every abelian group can be endowed with a threshold order which is not total, the nonabelian group introduced by Clifford cannot. This group is not finitely generated. We do not know if the free group on finitely many generators can be endowed with a threshold order which is not total; still we are able to prove that it can be equipped with a compatible semiorder which is not a weak order.

Our results are detailed in Section \ref{section:presentation}. Before, we  start with a section of prerequisites; some readers may want to skip this section except for the last subsection in which we introduce the  notion of threshold group which is the main object of study of this paper.  We include the proofs when they contribute to a better understanding. Otherwise, they are distributed in the six last   sections.

The  results described in this paper have been presented at the Combinatorial days of Sfax, May 28, 2013.
The first author thanks the organizers, Imed and Youssef Boudabbous, for their hospitality.
The authors are pleased to thank the referees  for their very detailed work, and   insight into the material.

\section{Prerequisite}
We recall the basic  notions of the theory of  ordered sets	that we will use, particularly the notions of interval orders and semiorders with their main properties. We do not restrict our attention to ordered sets, our investigation requires several properties of quasi-ordered sets.   Next, we recall the notions of  groups and ordered groups. We conclude with the notion of threshold group.
\subsection{Quasi-orders, orders, chains and antichains} We recall that a binary relation $\rho$ on a set $X$ is a \emph{quasi-order} if it is reflexive and transitive, in which case we say that the set $X$ is \emph{quasi-ordered} and the pair $P=(X, \rho)$ is a \emph{quasi-ordered set} (qoset). The dual of $P$ is the qoset $P^{d}=(X, \rho^{-1})$ where $\rho^{-1}$ is the opposite quasi-order defined by $x\rho^{-1}y$ if $y\rho x$.  As usual,  we denote a  quasi-order by the symbol $\leq$. Two elements  $x$ and $y$ of $X$ are \emph{comparable} with respect to $P$ if $x\leq y$ or $y\leq x$ and we set $x\sim y$ (despite the fact that this relation is not an equivalence); otherwise we say that $x$ and $y$ are \emph{incomparable} and we set $x\nsim y$. The   \emph{incomparability graph of $P$} is the graph $Inc(P)= (X, \nsim )$  whose edges are the pairs $(x,y)$ such that $x\nsim y$. The relation $\leq$ is \emph{total} if for every two elements $a, b\in X$ either $a\leq b$ or $b\leq a$ holds. If $\leq$  is an antisymmetric quasi-order, we call it  a \emph{partial order} (or an \emph{order}),  the set $X$ is \emph{partially ordered} and the pair $P=(X, \leq)$ is a  \emph{partially ordered set}  (poset for short). A total order is also called a  \emph{linear order} and the pair $P= (X, \leq)$ is a \emph{chain}. A  \emph{$n$-element chain} is a chain on a set of cardinality $n$.  A set of pairwise incomparable elements of a poset is called an \emph{antichain}. A \emph{$n$-element antichain} is an antichain of cardinality $n$.  A poset $P$ is \emph{bipartite} if it is the union of (at most) two antichains. Equivalently, every chain in $P$ has at most two elements. An element $x$ in a poset $P=(X, \leq)$  is \emph{isolated} if it is incomparable to every element of $X\setminus\{x\}$.

\subsection{Lexicographical sum} Let $I$ be a poset such that $|I|\geq 2$ and let $\{P_{i}=(X_i,\leq_i)\}_{i\in I}$ be a family of pairwise disjoint nonempty qosets that are all disjoint from $I$. The \emph{lexicographical sum} $\displaystyle \sum_{i\in I} P_{i}$ is the qoset defined on $\displaystyle \bigcup_{i\in I} X_{i}$ by $x\leq y$ if and only if
\begin{enumerate}[(a)]
\item There exists $i\in I$ such that $x,y\in X_{i}$ and $x\leq_i y$ in $P_{i}$; or
\item There are distinct elements $i,j\in I$ such that $i<j$ in $I$,   $x\in X_{i}$ and $y\in X_{j}$.
\end{enumerate}

The qosets $P_{i}$ are called the \emph{components} of the lexicographical sum and the poset $I$ is the \emph{index set}.

If $I$ is a totally ordered set, then $\displaystyle \sum_{i\in I}
P_{i}$ is called a \emph{linear sum}. On the other hand, if $I$ is an antichain, then $\displaystyle \sum_{i\in I} P_{i}$ is called a \emph{direct sum}. Henceforth we will use the symbol $\oplus$ to indicate direct sum.

Let $p$ and $q$ be two nonnegative integers. Denote by $p \oplus q$ the poset  direct sum of two chains on $p$ and $q$ elements respectively. The Hasse diagrams of the direct sums $2 \oplus 2$ and $3 \oplus 1 $ are depicted in Figure \ref{exception} (a) and (b) respectively.

\subsection {Posets versus qosets, initial segments, convex subsets, autonomous subsets} Let $P=(X, \leq)$  be a qoset. The relation $\equiv$ defined by $x\equiv y$ if $x \leq y$ and $y\leq x$ is an equivalence relation, whereas the relation $<$ defined by $x<y$ if $x\leq y$ and $y\not \leq x$ is a \emph{strict-order}, that is a  transitive and irreflexive relation. For an element $x\in X$ we denote by $\overline{x}$ the equivalence class of $x$ with respect to $\equiv$ and $X/\equiv$ the set of equivalence classes. Then the set $X/\equiv$ can be induced with a partial order $\overline{\leq}$ so that $\overline{x}\, \overline{\leq}\, \overline{y}$ if $x\leq y$. Let $P/\equiv$ be the ordered set $(X/\equiv, \overline{\leq})$. The quotient map $p: X\rightarrow X/\equiv$ defined by setting $p(x):= \overline{x}$ is order preserving, that is if $x\leq y$, then $p(x)\, \overline{\leq}\, p(y)$.

Let $P= (X, \leq)$ be a qoset. A subset $Y$ of $X$ is a \emph{final segment} of $P$ if $x\in Y$ and $x\leq y$ implies $y\in Y$. This is an \emph{initial segment } of $P$ if it is a final segment for the opposite qoset $P^{d}$. The subset $Y$ is  a \emph{convex} subset of  $P$  if for all $y,y'\in Y$ the set $\{z\in X : y\leq z\leq y'\}$ is a subset of $Y$. The set $Y$ is \emph{autonomous} in $P$ if for every $y,y'\in Y$ and $x\in X\setminus Y$, $x\leq y$ is equivalent to $x\leq y'$ and  $y\leq x$ is equivalent to $y'\leq x$. The empty set, the singletons in $X$ and the whole set $X$ are autonomous and are said to be {\it trivial}. The qoset $P$ is \emph{prime} if its only autonomous subsets are trivial. The notion of autonomous set was introduced in \cite{fraisse} and \cite{gallai}; since then a huge literature on this notion proliferated, see  \cite{ehrenfeucht-al}.

If $\leq$ is an order, autonomous subsets are convex. If $\leq$ is  total, convex subsets are autonomous. Hence, if $\leq$ is a total order, convex and autonomous sets coincide, and are usually called \emph{intervals}. If $P$ is the lexicographical sum of qosets, the components are autonomous.
Each element of $X/\equiv$ is an autonomous subset in $P$, hence $P$ is the lexicographical sum of the classes by the quotient $P/\equiv$. Each autonomous subset in $P/\equiv$ is convex and its inverse image by the quotient map $p$ is a convex autonomous subset in $P$. A subset $Y$ of $P$ is convex if and only if $Y$ is the inverse image under $p$ of some convex subset of $P/\equiv$. We recall that for every $x\in X$ there is a	 largest autonomous subset containing $x$ which is an  antichain, resp. a complete relation. This elementary fact is frequently used (see \cite{Pouzet-Thiery2}). It  is a consequence of the fact that an union of autonomous subsets containing an element  is autonomous (see \cite {ehrenfeucht-al} p.48).   Note that  necessarily,  one of these largest sets is a singleton.

\subsection{Embeddability, weak-orders, semi-orders, interval orders}\label{subsection2}
Given two posets $P=(X,\leq)$ and $P'=(X',\leq')$,  we say that $P$ \emph{embeds} into $P'$ if there exists a one-to-one map from $X$ to $X'$ such that $f(x)\leq' f(y)$ if and only if $x\leq y$ for all $x,y\in X$. Posets which do not embed the direct sum $1\oplus 1$ are just chains.
 Posets which do not embed $1\oplus 2$ are called \emph{weak orders}. As it is immediate to see and well known, a poset $P:= (X, \leq)$ does not embed $1\oplus 2$ if and only if the binary relation defined on $X$ by "$x$ is incomparable to $y$ or $x=y$" is an equivalence relation. For instance, a chain and an antichain (the order is the equality relation) are weak orders. More generally, every weak order is a lexicographical sum of antichains indexed by a chain.
 We now introduce  the notions on which this paper is built.
\begin{definition}
\begin{enumerate}[(a)]
\item The order of a poset is an \emph{interval order} if it does not embed $2 \oplus 2$.
\item The order of a poset is a \emph{semiorder} if it does not embed $2 \oplus 2$ nor $3 \oplus 1 $.
\end{enumerate}
\end{definition}

\begin{figure}[ht]
\begin{center}
 \leavevmode \epsfxsize=1.6in \epsfbox{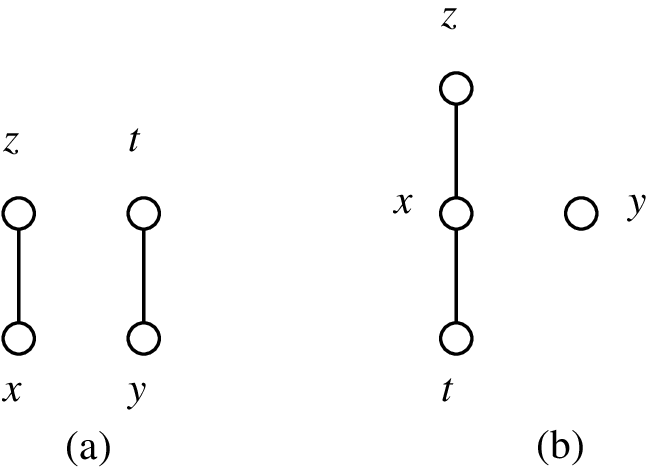}
\end{center}
\caption{}
\label{exception}
\end{figure}

Interval orders were considered in \cite{fi} and \cite{wiener} in relation to the theory of measurement. Semiorders were introduced and applied in mathematical psychology by Luce \cite{luce56}. For a wealth of information about interval orders and semiorders the reader is referred to \cite{aleskerov-al, bgr, bridges, pirlotvincke, fi-book}.
The name interval order comes from the fact that if the order of a poset $P$ is an interval order, then $P$ is isomorphic to a subset $\mathcal J$ of the set $\Int(C)$ of nonempty intervals of a totally ordered set $C$, ordered as follows: if $I, J\in\Int(C)$, then
\begin{equation}\label{equ:interval}
I<J \mbox{ if } x<y \mbox{ for every } x\in I \mbox{ and every } y\in J.
\end{equation}

This fact is due to Fishburn (see Theorem 4 of \cite{fi}) in the case where the equivalence relation associated to the absence of strict preference is countable and to Bogart \cite{bogart} in the general case. See also Wiener \cite{wiener}.

The Scott and Suppes Theorem \cite{scottsuppes} states that the order of a \emph{finite} poset $P$ is a semiorder if and only if $P$ is isomorphic to a collection of intervals of length 1 of the real line, ordered as in Statement (\ref{equ:interval}). Extensions of Scott and Suppes Theorem to infinite posets have been considered, notably the representation in the real numbers, see \cite{candeal, candeal1}, \cite {luce87}, and \cite{manders}.
Another extension in presented in Theorem \ref{thm:1} below.

\subsection{The quasi-orders $\leq_{pred}$ and $\leq_{succ}$ \label{succ and pred}} 
Let $P= (X, \leq)$ be a qoset. Set $x\leq_{pred}y$ if $z<x$ implies $z<y$ for all $z\in X$ and set $x\leq_{succ}y$ if $y<z$ implies $x<z$ for all $z\in X$. The relations $\leq_{pred}$ and $\leq_{succ}$ are called the \emph{traces} of $<$ or $\leq$.

 The relations $\leq_{pred}$ and $\leq_{succ}$ associated to the quasi-order $\leq$ are quasi-orders (even in the case where $\leq $ is an order relation) containing $\leq$. Hence  $\leq$ is an order  whenever $\leq_{pred}$ or $\leq_{succ}$ is an order. Interval orders and semiorders have been characterized in terms of the quasi-orders $\leq_{pred}$ and $\leq_{succ}$ associated with a given order $\leq$. Indeed, as indicated in Lemma \ref{intervalorder}, an order $\leq$ is an interval order iff $\leq_{pred}$ is a total quasi-order, equivalently $\leq_{succ}$ is a total quasi-order. Furthermore, $\leq$ is a semiorder iff the intersection of $\leq_{pred}$ and $\leq_{succ}$ is a total quasi-order (Lemma \ref{critical}). To the quasi-order $\leq_{pred}$ corresponds the equivalence relation $\equiv_{pred}$ defined by $x\equiv_{pred} y$ if $x\leq_{pred}y$ and $y\leq_{pred} x$. Similarly, to the quasi-order $\leq_{succ}$ corresponds the equivalence $\equiv_{succ}$.

The strict orders associated to $\leq_{pred}$ and $\leq_{succ}$ extend the strict order associated to $\leq$, that is:
\begin{equation}\label{equa:1}
x< y\Rightarrow x<_{pred}y \; \text{and} \; x<_{succ}y
\end{equation}
for all $x,y \in X$.
Indeed, let $z< x$ then, since $<$ is a strict order, $z<y$ hence $x\leq_{pred} y$. Necessarily, $y\not \leq_{pred} x$. Otherwise, from $x<y$ and $y\leq_{pred}x$ we would have $x<x$ which is impossible.

We recall the following result (see Theorems 2 and 7 of \cite{rabinovitch}).
\begin{lemma} \label{intervalorder}An order $\leq$ on a set $X$ is an interval order if and only if
$\leq_{pred}$ is a total quasi-order; equivalently $\leq_{succ}$ is a total quasi-order.
\end{lemma}

We give a characterization of semiorders which is slightly different from the one given by Rabinovitch (see Theorem 7 \cite{rabinovitch}).

Let $P:= (X, \leq)$ be a poset. For an element $u\in X$ set $D(u):=\{v\in X : v< u\}$ and $U(u):=\{v\in X : u< v\}$. We recall that an incomparable pair $(x,y)$ of elements is {\it critical} if $D(x)\subseteq~D(y)$ and $U(y)\subseteq U(x)$ (see \cite{trotterbook}). Denote by $crit(P)$ the set of all critical pairs of $P$. The importance of this notion in the case of interval orders is pointed in  \cite{estevan-al}.

\begin{lemma} \label{critical} An order $\leq$ on a set $X$ is a semiorder if and only if the intersection of the quasi-orders $\leq_{pred}$ and  $\leq_{succ}$  is a total quasi-order. Equivalently, for every incomparable elements $a$ and $b$, the pair $(a,b)$ or the pair $(b,a)$ is critical. \end{lemma}

\begin{proof}

We  use the fact that for every
order $\leq$ (semiorder or not), the intersection of the quasi-orders $\leq_{pred}$ and $\leq_{succ}$ is equal to the union of $\leq$ and $crit(P)$. From this fact  follows that this intersection is a total quasi-order iff for every two incomparable elements  $a$ and $b$ either $(a,b)$ or $(b,a)$ are critical.  Let us prove that this condition on pairs holds if $\leq$ is a semiorder. Suppose the contrary. Let $a,b$ be two incomparable elements and let $t\in X$  witnessing the fact that $(a,b)$ is not critical. We may assume without loss of generality that  $t<a$ and $t\nless b$. Since $(b,a)$ is not critical, we infer that there exists $t'$ such that $t'<b$ and $t'\nless a$ or $a<t'$ and $b\nless t'$. In the former case we have that $\{t,a,t',b\}$ is isomorphic to $2\oplus 2$ and in the latter case we have that $\{t,a,t',b\}$ is isomorphic to $3\oplus 1$. In both cases we obtain a contradiction since the order $\leq$ is a semiorder.
Now, suppose that for every incomparable elements $a$ and $b$, the pair $(a,b)$ or the pair $(b,a)$ is critical.  We claim that the order is a semiorder. Indeed, let $x,y,z,t$ be such that $x<y$ and $z<t$ and $x$ is incomparable to $z$. Then $(x,z)$ or $(z,x)$ must be  a critical pair and therefore $x<t$ or $z<y$. This shows that there is no $2\oplus 2$ in $(X,\leq)$. Similarly we prove that there is no $3\oplus 1$ in $(X,\leq)$. Hence, $X$ is a semiorder.  \end{proof}

As a consequence  of Lemmas \ref {intervalorder} and \ref{critical}, an order is a semiorder whenever the quasi-orders $\leq_{pred}$ and $\leq_{succ}$ are equal and total.
But in general, these quasi-orders are not equal as this can be seen on  the direct sum $2\oplus 1$.

This leads to the following definition.

\begin{definition} We say that an order $\leq$ is a \emph{threshold order} if $\leq_{pred}$ and $\leq_{succ}$ are total orders and equal.
\end{definition}
Since $\leq_{pred}$ and $\leq_{succ}$ coincide with $\leq$ when  $\leq$ is a total order, total orders are threshold orders.

\subsection{Ordered groups}Let $G=(X,+)$ be a group. We do not require the group to be abelian (the role of commutativity will be emphasized when needed). Still, we denote the group operation by $+$, the neutral element by $0$ and by $-x$ the inverse of any element $x$. A \emph{right translation} of $G$ is map $t$ from $X$ to $X$ such that there exists an element $g$ of $X$ so that $t(x)=x+g$ for all $x\in X$. Similarly we define \emph{left translations}. A \emph{translation} is either a right or a left translation. A subset $F$ of $G$ is \emph{normal} if $-x+F+x\subseteq F$ for every $x\in G$. Note that a subset $F$ is normal whenever $-F$ is normal; also, if  $G$ is abelian every subset of $G$ is normal. If a  subgroup $H$ of $G$ is normal, then we may define the  \emph{quotient group} $G/H$ made of translates of $H$ called the \emph{cosets} of $H$ (when needed, see Subsection \ref{normal} for more about normality).

If $\rho$ is a binary relation on a set $X$ and $G=(X,+)$ is a group, we say that $\rho$ is \emph{compatible} if for all $x,x',y,y'\in G$, $x\rho x'$ and $y\rho\, y'$ imply $(x+y)\rho\, (x'+y)$ and $(x+y) \rho\, (x+y')$. If $\rho$ is a quasi-order, this condition amounts to $(x+y)\rho\, (x'+y')$, it expresses that left and right translations on $G$ preserve $\rho$. Alternatively,
\begin{equation}\label{equ:preserve}
x\rho\, y\; \text{iff}\; 0\, \rho\, (y-x)\; \text{iff}\; 0\, \rho\, (-x +y)\; \text{iff} -y\, \rho -x
\end{equation}
for all $x,y\in G$.

We recall that a \emph{quasi-ordered group} is a group equipped with a quasi-order  which is compatible with the group operation.  If the quasi-order is an order, the group is an ordered group.  Throughout, $G=(X,+\leq)$ will denote a  quasi-ordered group. An element $x$ of $X$ is called \emph{positive} if $0 \leq x$. The set of positive elements is often denoted with $G^+$, and it is called the \emph{positive cone} of $G$. So we have $a\leq b$ if and only if $-a+b \in  G^+$. For a group $G$, the existence of a positive cone specifies a quasi-order on $G$. A group $G$ is an ordered group if and only if there exists a subset $C$ (which is $G^+$) of $G$ such that:
\begin{itemize}
  \item $0\in C$.
  \item If $a \in C$ and $b \in C$, then $a+b \in C$.
  \item If $a \in C$, then $-g+a+g \in C$ for all $g\in G$.
  \item If $a \in C$ and $-a \in C$, then $a=0$.
\end{itemize}

Some general properties of ordered groups, notably about convexity are given in Section \ref {subsection:convex}.
See the books by Fuchs \cite{fuchsbook} and Glass \cite{glassbook} for a wealth of information about ordered groups.
\subsection{Threshold groups}
The following lemma shows that in a quasi-ordered group the quasi-orders $\leq _{pred}$ and $\leq _{succ}$ are equal. It has  a very simple proof that we give below.

\begin{lemma}\label{pred=suc}If $\leq$ is a compatible quasi-order on a group $G$ then
\begin{enumerate}[(a)]
 \item $\equiv$, the equivalence relation associated to the quasi-order $\leq$, is a compatible equivalence relation and in particular $\overline 0$, the equivalence class of $0$, is a normal subgroup;
 \item $<$ is a compatible strict order;
 \item $\leq _{pred}$ is a compatible quasi-order;
 \item $\leq_{pred}\, =\, \leq_{succ}$.
\end{enumerate}
\end{lemma}
\begin{proof}
\noindent $(a)$ and $(b)$ are immediate, because translations preserve the quasi-order.

\noindent $(c)$ Let $x,y\in G$. Suppose that $x\leq_{pred}y$. We prove $x+a\leq_{pred} y+a$ for every $a\in G$. Let $z<x+a$ and set $z':=z+(-a)=z-a$. Since $z=z'+a$ we have $z'+a<x+a$. Hence, by adding -a on the right, we have $z'<x$. Thus $z'<y$ i.e., $z+(-a)<y$ amounting to $z<y+a$ proving that $x+a\leq_{pred} y+a$ as claimed. Similarly we prove that $a+x\leq_{pred}a+y$.

\noindent $(d)$
Since $\leq_{pred}$ is compatible we deduce from the equivalence stated in (\ref{equ:preserve}) that $x\leq_{pred} y$ iff $0\leq_{pred} y-x$ iff $0\leq_{pred} -x +y$ iff $-y\leq_{pred} -x$. We now prove that for every $a$,
$0\leq_{pred} a$ iff $-a \leq_{succ}0$. Indeed, assume that $0\leq_{pred} a$ and let $t>0$. Then $-t<0$ and therefore $-t<a$, that is, $-a <t$, proving that $-a \leq_{succ}0$. Since $0\leq_{pred} a$ iff $-a \leq_{succ} 0$ we infer that $\leq_{pred}=\leq_{succ}$.
\end{proof}

As indicated  in $(d)$ of Lemma \ref{pred=suc},  if $\leq$ is a compatible quasi-order on a group $G$, $\leq_{pred}$ and $\leq_{succ}$ are equal. Thus, the equivalence relations $\equiv_{pred}$ and $\equiv_{succ}$ associated to $\leq_{pred}$ and $\leq_{succ}$ are equal, and hence compatible with the group operation. Consequently,
\begin{equation}\label{KdeG}
K(G), \mbox{ the equivalence class of } 0 \mbox{ with respect to } \equiv_{pred}, \mbox{ is a normal subgroup of } G.
\end{equation}
In particular, $K(G)=\{0\}$ iff $\leq_{pred}$ is an order. Note also that if $\leq$ is a total quasi-order, $\equiv_{pred}$ coincide with $\equiv$.

Lemma \ref{pred=suc} introduces  to the central definition of this paper:

\begin{definition} An ordered group $G= (X, +, \leq)$ is a threshold group if $\leq$ is a threshold order.
\end{definition}

The order $\leq$ of an ordered group is a threshold order iff $\leq_{pred}$ is a  total  order, or equivalently, $\leq_{succ}$ is a total  order. Indeed, according to  Lemma \ref{pred=suc}, on an ordered group the quasi-orders $\leq_{pred}$ and $\leq_{succ}$ coincide.

An ordered group $G$ is a threshold group iff $\leq$ is a semiorder and $K(G) = \{0\}$.
More generally, if the order of an ordered group $G$ is a semiorder, then up to a quotient
by some unordered normal subgroup, namely $K(G)$, this is a threshold group.
Our aim in this paper is to describe the structure of threshold group.

\section{Presentation of the main results}\label{section:presentation}

An illustration of the interactions between order and group operation in an ordered group, which motivates this paper, is given by Theorem \ref {2+2,3+1} and its corollary.

\subsection{Posets embeddable in ordered groups}

We say that a poset $P=(X,\leq)$ \emph{embeds} into an ordered group $G=(Y,+,\leq')$ if $(X,\leq)$ embeds into $(Y,\leq')$.

\begin{theorem} \label {2+2,3+1} Let $n\geq 2$ be an integer and $p, q$ be any two positive integers such that $n= p+q$. Then, an ordered group $G$ embeds $1\oplus n$ if and only if it embeds $(q+1) \oplus p$.
\end{theorem}
\begin{proof}
Suppose that $G$ embeds $1 \oplus n$. Let $A$ be a subset of $G$ isomorphic to $1 \oplus n$, say $A:= \{x_{0},y_0,\dots ,y_{n-1}\}$, with $y_0< \dots <y_{n-1}$ and $x_0$ incomparable to all $y_i$ for $i<n$. We may suppose $x_0=0$. Indeed, the translation $t_{-x_{0}}$, defined by $t_{-x_{0}}(u):= u-x_{0}$ for all $u\in G$ is a bijective map preserving the ordering, hence it maps $A$ to an isomorphic copy $A'$, and maps $x_0$ to $0$.
Let $X:= \{y_{p-1}-y_{n-1}, \dots, y_{p-1}-y_i, \dots, y_{p-1}-y_{p}, 0\}$ and $Y:= \{ y_{0},\dots, y_{p-1}\}$. Then $B:= X\cup Y$ is the direct sum of the chains $X$ and $Y$, hence it is isomorphic to $(q+1) \oplus p$. Indeed, if some $x\in X$ is comparable to some $y\in Y$ then $x\leq y$, otherwise $y<0$, which is impossible. In that case, we have $y_{p-1}-y_{n-1}\leq x \leq y \leq y_{p-1}$, from which it follows that $-y_{n-1} \leq 0$ , that is, $0\leq y_{n-1}$, which is impossible.\\
Conversely, suppose that $G$ embeds $(q+1) \oplus p$. Let $B:=X\cup Y$ with $X:= x_{0} < \dots <x_{q}$, $Y:= y_0<\dots < y_{p-1}$ and every element of $X$ is incomparable to every element of $Y$. We may suppose that $x_{q}= 0$ (otherwise, take the image of $B$ by the translation $t_{-x_{q}}$). Then $\{0\}\cup Z$ where $Z= Y\cup \{y_{p-1}-x_{q-1}, \dots, y_{p-1}-x_i, \dots, y_{p-1}-x_{0} \}$ is isomorphic to $1\oplus n$. Indeed, if $0$ is comparable to some element $y$ of $Z$, necessarily $y\not \in Y$ hence $y=y_{p-1}-x_i$ for some $i$ and in fact $y_{p-1}-x_i<0$ for some $i$ which is impossible.
\end{proof}

An immediate consequence of Theorem \ref{2+2,3+1} is this.

\begin{corollary}\label{cor:semiorder}
The order of an ordered group is an interval order if and only if it is a semiorder.
\end{corollary}

Note that Lemmas \ref{intervalorder}, \ref{critical} and \ref{pred=suc} yield another  proof of Corollary \ref{cor:semiorder}. Indeed,  let $G$ be a group equipped with a compatible  quasi-order $\leq$. Then, by Lemma \ref{pred=suc},  the quasi-orders $\leq_{pred}$ and $\leq_{succ}$ are compatible and in fact equal. According to Lemma \ref{intervalorder}, $\leq $ is an interval order iff $\leq _{pred}$ is a total quasi-order. According to Lemma \ref{critical}, $\leq$ is a semiorder provided that the intersection of $\leq_{pred}$ and $\leq_{succ}$ is a total quasi-order. Hence, the order on an ordered group is an interval order iff this is a semiorder.

From this corollary, we are lead to examine the order structure of groups equipped with a semiorder.

We start with an example. Ordering the intervals of the real numbers of the form $[a, a+1[$ by the order defined in Equation (\ref{equ:interval}) of Subsection \ref{subsection2}  amounts to order  the reals by $x \leq_1 y$ if $x=y$ or $1\leq y-x $. The additive group of the reals equipped with
this order is an ordered group, the order is a semiorder, in fact a threshold order since $\leq_{1  {pred}}$ coincide with the order $\leq$ on the reals.
Replacing $1$ by a positive real $\alpha$, the ordering we obtain, say $\leq_{\alpha}$, is different, but the  ordered set is isomorphic to the previous one.

This construction generalizes: 

The  \emph{center} $Z(G)$ of  a group $G$  is the set of elements that commute with every element of~$G$.
\begin{proposition}\label{center} Let $G:=(X,+,\preceq)$ be a totally ordered group and let $\alpha\in Z(G)$ be positive. Set
\[x\leq_{\alpha} y \mbox{ if and only if } x=y \; \text{or} \; \alpha\preceq y-x\]
and
\[x\leq_{\check{\alpha}} y \mbox{ if and only if } \; x=y \; \text{or}\; \alpha \prec y-x. \]
Then $G_{\alpha}:= (X, +, \leq_{\alpha})$ and $G_{\check{\alpha}}:= (X, +, \leq_{\check{\alpha}})$ are threshold groups.
\end{proposition}

The proof is straightforward.  Using the fact that $\alpha\in Z(G)$, one proves first that   $\leq_{\alpha}$ and $\leq_{\check{\alpha}}$ are compatible order relations. Next,  one   proves that the quasi-orders ${\leq_{\alpha}}_{ pred}$ and ${\leq_{\check{\alpha}}}_{ pred}$  associated to ${\leq_{\alpha}}$ and ${\leq_{\check{\alpha}}}$ are equal to $\preceq$. We will see below how to deduce this result from the general construction given in Proposition \ref{prop:thresholdgroup}.

In the above examples, $\alpha$ is a \emph{threshold}; to distinguish the first example from the second, we say that in the first example  the threshold is \emph{attained}.

We prove the following extension of Scott-Suppes's representation of semiorders.

\begin{theorem}\label{thm:1}
Let $P$ be a poset. The following propositions are equivalent.

\begin{enumerate}[(i)]
\item The order on $P$ is a semiorder.
\item $P$ is isomorphic to a collection $\mathcal J$ of intervals of some chain which are pairwise incomparable with respect to inclusion and ordered by the order relation defined by Statement (\ref{equ:interval}).
\item $P$ is isomorphic to a collection $\mathcal J$ of intervals of some totally ordered abelian group $G$ ordered by the order relation defined in Statement (\ref{equ:interval}); these intervals being of the form $[x, x+ \alpha[= \{y\in G: x\leq y<x+ \alpha\}$ for some positive $\alpha$.
\item $P$ is embeddable into an abelian threshold group with some attained threshold $\alpha$.
\item $P$ is embeddable into a threshold group.
\end{enumerate}
\end{theorem}
\begin{proof} Implication $(i) \Rightarrow (iv)$ follows from Scott-Suppes Theorem and the Compactness Theorem of first order logic. To show that, we use the "diagram method" of Robinson (see Chapter 5 of \cite{shoenfield} for the logic setting).  Let $P$ be a semiorder and let  $\mathcal L$ be the first order language consisting of a binary predicate symbol $\leq$, the symbol of a binary operation $+$, constant symbols $0$, $\alpha$, $(a_x)_{x\in P}$. We define a set $\mathcal A$ of sentences such that a model (if any)  is an abelian  threshold group $G$ such that the group operation $+^G$, the order $\leq^{G}$ on the group, the element $\alpha^G$ of $G$ are the interpretation of $\leq$, $+$ and $\alpha$ and the map from $P$ into $G$ which associates to $x$ the interpretation $a_x^G$ of $a_x$ is an embedding. The sentences are chosen is such a way that every finite subset of sentences will be consistent via Scott-Suppes Theorem. Compactness theorem will ensure that the whole set is consistent, thus has a model, say $G$. The interpretation of the constants provides an embedding into $G$.\\
$(iv) \Rightarrow(iii)$ Let $f$ be an embedding of $P$ into $G$. Let $\mathcal J:= \{[f(x), f(x)+\alpha[: x\in P\}$ and $\overline f: P\rightarrow \mathcal J$ defined by setting $\overline f (x):= [f(x), f(x)+\alpha[$. Then, as it is easy to check, $\overline f$ is an embedding.\\
$(iii)\Rightarrow (ii)$ Obvious.\\
$(ii) \Rightarrow (i)$ Observe that if two intervals are incomparable, with respect to the order defined on intervals, then they must have a nonempty intersection. Hence, if $3\oplus 1$ is embeddable into a collection of intervals, the image of the 1-element chain must include the image of the intermediate element of the 3-element chain. This contradicts the fact that $\mathcal J$ is an antichain.\\
$(iv)\Rightarrow (v) \Rightarrow (i)$ Obvious.
\end{proof}

We illustrate this result first with a new notion of embeddability between posets and the corresponding equivalence: two posets being equivalent if every group embedding one of these posets embeds the other (see Section \ref{section:comparison} and Item (2) of Proposition \ref {equivalence class}). Next, we illustrate it  with the notion of dimension.
Let $P:=(X,\leq)$ be a poset. A \emph{linear extension} of $\leq$ (or of $P$) is a total order $\preceq$ on $X$ such that $x \preceq y$ whenever $x\leq y$. The \emph{dimension} of $P$, denoted by $dim(P)$, is the least
cardinal $\kappa$ such that there exists a family $\mathcal{R}$ of $\kappa$ linear extensions $\preceq_{i}$, $i<\kappa$,  of $\leq $ so
that $x\leq y$ iff $x\preceq_i y$ for all $i<\kappa$ \cite{dm}.
The dimension of finite posets is finite, and the dimension of finite interval orders is unbounded (see \cite{{brt}} and \cite{fhrt}). On the other hand and according to Rabinovitch \cite{rabinovitch}, finite semiorders have dimension at most $3$. This extends to infinite semiorders (via the Compactness Theorem of first order logic). In Section \ref{subsection:dimension} we give an effective proof that threshold groups, with attained threshold, have dimension at most $3$.  Thus, Rabinovitch's result follows from Theorem \ref{thm:1} and Proposition~\ref{thresholddimension}.

Let us illustrate the scope of Theorem \ref {2+2,3+1}. Let $n$ be a positive integer. For $n=1$ or $n=2$, Theorem \ref{2+2,3+1} says nothing.  But, as it is immediate to see by using translations, \emph{an ordered group $G$ does not embed $1\oplus n$ if and only if the set $inc(0)$  of elements of $G$ which are incomparable to $0$  does not embed an $n$-element chain}. For a poset, not embedding an $n$-element chain amounts to be a union of less than $n$ antichains \cite{mirsky}; for $n=3$, these posets are said to be \emph{bipartite}. We are lead to the following:

\begin{problem}\label{prob1} Describe the orders of ordered groups $G$ for which the set $inc(0)$ is a union of less than $n$ antichains.
\end{problem}

Ordered groups which do not embed $1\oplus 1$ are totally ordered groups. Ordered groups whose order is a weak order are easy to describe.
Indeed, we prove in Subsection \ref{subsection:claim:weakorder} that:%
 \begin{proposition}\label{claim:weakorder}Let $G:=(X,+,\leq)$ be an ordered group and $inc(0)$ be the set of elements of $X$ incomparable to $0$. The following assertions are equivalent:
 \begin{enumerate}[(i)]
\item The order of $G$ is a weak order.
\item $inc(0)$ is an antichain in $(X,\leq)$.
\item $inc(0)\cup \{0\}$ is a subgroup of $G$.
\end{enumerate}

If any of the above conditions hold, then $H:=inc(0)\cup \{0\}$ is a normal subgroup of $G$, the quotient group $G/H$ is totally ordered and the order on $G$ is the lexicographical sum of copies of $H$ indexed by $G/H$.
\end{proposition}

A \emph{torsion} element of a group $G$ is any element $x$ such that $nx=0$ for some positive integer $n$. We recall that if $G$ is abelian, the set $T(G)$ of torsion element is a subgroup and  the quotient $G/T(G)$ is totally orderable.

\begin{corollary}\label{cor:existenceweak}A group   can be equipped with a weak order distinct from the equality relation (hence with a nontrivial semiorder) iff it admits  a totally ordered quotient by some proper subgroup. If $G$ is abelian, these  conditions amount to the fact that the subgroup $T(G)$ of  torsion elements  is a proper subgroup.
\end{corollary}

A strengthening of this corollary is given in Proposition \ref {prop:existence}.

The case $n=3$ of Theorem \ref {2+2,3+1} corresponds to ordered groups for which the order is a semiorder. We will describe the compatible semiorders and particularly the threshold orders on groups. In the vein of Proposition \ref{claim:weakorder}, we will prove in Section \ref{section:basic} the following result:

\begin{theorem} \label{claim:interval} Let $G$ be an ordered group. Then
the order is a semiorder if and only if $inc(0)$ is bipartite.
Furthermore, the following properties are equivalent:
\begin{enumerate}[{(i)}]
\item $G$ is a threshold group;
\item $inc(0)$ has no isolated elements and is bipartite;
\item $inc(0)$ is prime and bipartite.
 \end{enumerate}

\end{theorem}

A threshold order is a weak order if and only it is a total order. This particularly applies to ordered groups. Indeed, $G:=(X, +, \leq)$ is a threshold group and $(X, \leq)$ is a weak order if and only if $\leq$ is total (see Proposition \ref{claim:weakorder}).

For $n>3$, a particular aspect of Problem \ref{prob1} is: \emph{which posets necessarily embed into ordered groups $G$ such that $inc(0)$ is the union of $n-1$ antichains}?

This leads us to introduce in Section \ref{section:comparison} a quasi-order on the class of posets which extends the embeddability relation and to describe some of the equivalence classes associated to this quasi-order.

\subsection{Another description of  compatible semiorders and threshold orders}\label{subsection:threshold}

In the next proposition we characterize compatible semiorders on groups and threshold groups in terms of an auxiliary total quasi-order and a final segment.

\begin{proposition} \label{prop:thresholdgroup} Let $G:=(X,+,\leq)$ be an ordered group.
\begin{enumerate}
\item The order $\leq$ is a semiorder if and only if there is a compatible total quasi-order $\preceq$ on $G$, a normal final segment $F$ of $(X, \preceq)$ not containing $0$ such that $x<y$ in $G$ if and only if $y-x\in F$.

\item The order $\leq$ is a threshold order if and only if there is a compatible total order $\preceq$ on $G$, a nonempty normal final segment $F$ of $(X, \preceq)$ not containing $0$ such that:
 \begin{enumerate}[{(a)}]
\item $x<y$ in $G$ if and only if $y-x\in F$;
\item $I:= G\setminus (-F\cup F)$ is not a union of cosets of a normal convex subgroup of $(X, +, \preceq)$ distinct from $\{0\}$.\\
\end{enumerate}

When conditions $(a)$ and $(b)$ are satisfied, the auxiliary order $\preceq$ coincides with $\leq_{pred}$.
\end{enumerate}

\end{proposition}

This proposition yields another proof of Proposition \ref{center}. Indeed, let $\alpha$ be a positive element of a totally ordered group $G=(X, +, \preceq)$. The sets
 $F_{\alpha}:= \{ x\in X: \alpha \preceq x \}$ and $F_{\check{\alpha}}:=\{ x\in X: \alpha \prec x \}$  are final segments of positive elements of $G$. If $\alpha\in Z(G)$ then these final segments are normal subsets of $G$ (see Item (3)  of Subsection \ref{normal} if needed). Hence, from Item(1) of Proposition \ref{prop:thresholdgroup}, $\leq_{\alpha}$ and $\leq_{\check{\alpha}}$ are two compatible semiorders. In order to see that they are thresholds orders,  let $I_\alpha:=X\setminus (-F_{\alpha}\cup F_{\alpha})$ and $I_{\check{\alpha}}:=G\setminus (-F_{\check{\alpha}}\cup F_{\check{\alpha}})$.
Note that $I_\alpha=\{v\in X : -\alpha \prec v \prec \alpha\}$ and $I_{\check{\alpha}}=\{v\in X : -\alpha \preceq v \preceq \alpha\}$. We only need to check that $I_\alpha$ and $I_{\check{\alpha}}$ satisfy condition $(b)$ of Proposition \ref{prop:thresholdgroup}. The proofs being the same, let $I\in \{I_\alpha, I_{\check{\alpha}}\}$ and $K$ be a normal convex subgroup of $G= (X, +, \preceq)$ included into $I$. We claim that if $K\not =\{0\}$ then $I$ cannot be a union of cosets of $K$. Let $K+\alpha$ be the coset
containing $\alpha$. Since $K$ is totally ordered, it is torsion free. Being distinct from $\{0\}$, it has no least and largest element. Hence $K+\alpha$ has no least and largest element and therefore it meets both $I$ and $F$ and our claim follows.\hfill $\Box$

 The fact that a group $G$ admits a compatible total order imposes that $G$ be torsion-free (hence, nontrivial subgroups are infinite). Since every abelian and torsion-free group can be totally ordered \cite{levi1} we infer that:

\begin{corollary}\label{cor:threshold-abelian}
Every torsion-free abelian group can be equipped with a compatible threshold order which is not total.
\end{corollary}

\begin{figure}[ht]
\begin{center}
 \leavevmode \epsfxsize=3.5in \epsfbox{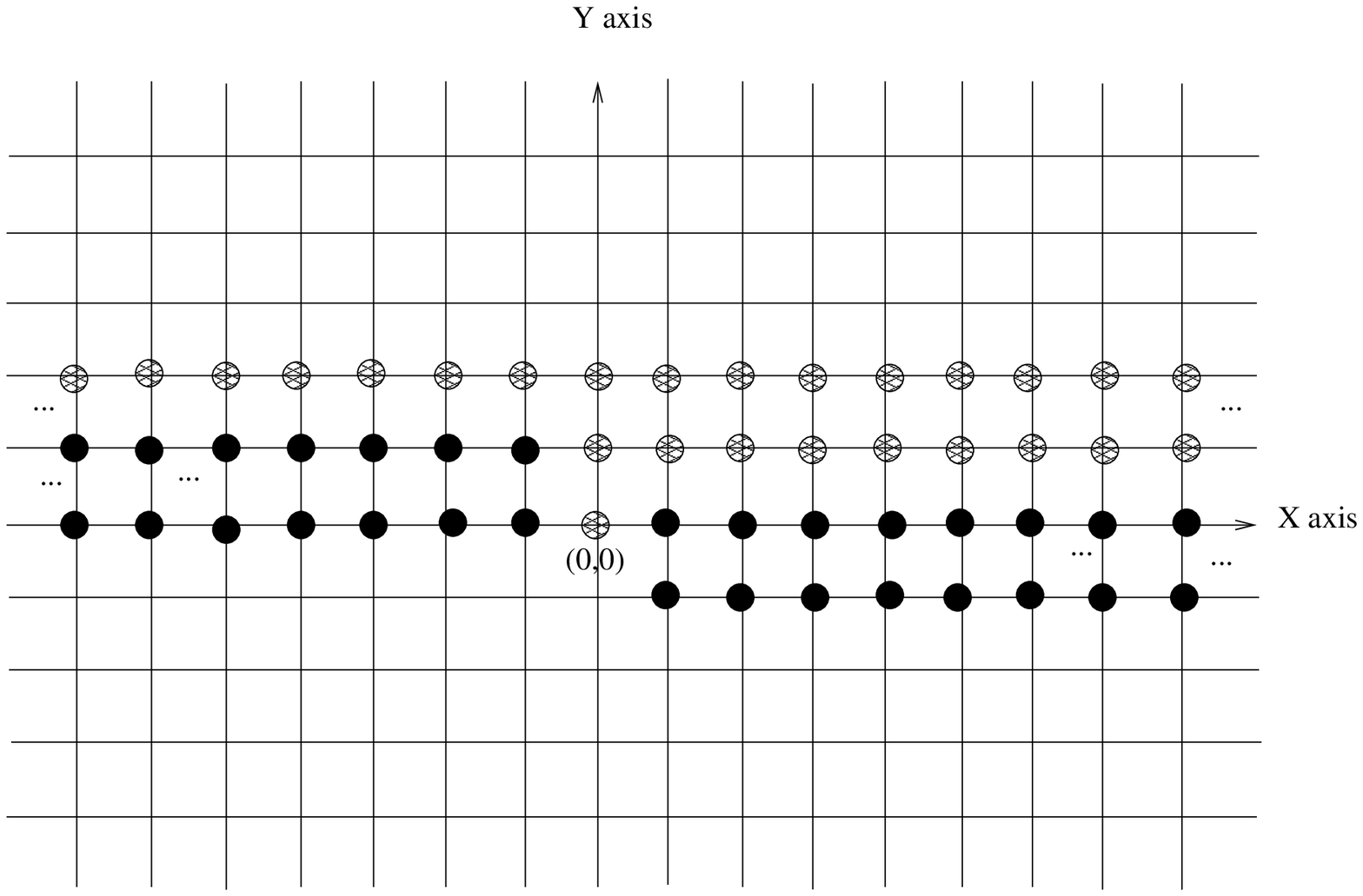}
\end{center}
\caption{An example of a threshold order on $\ZZ^2$. The set of bold points represents $inc((0,0))$. The set of hatched points and every other point above represents the positive cone.}
\label{grid}
\end{figure}

Different  total orders which are compatible with the group operation will lead to different semiorders. Compatible total orders on torsion free abelian groups and notably on   groups of the form $\ZZ^n$ have been  described (see Chapter 6 of \cite {kopytov} and \cite{robbiano}).
The case of the additive group of real numbers is interesting for our purpose. According to Proposition \ref{prop:thresholdgroup}, if $\leq$ is a semiorder such that $\leq_{pred}$ coincide with the order $\preceq$ on the reals then $\leq$ is one of the orders  $\leq_{\alpha}$ or     $\leq_{\check\alpha}$ (for some real $\alpha$) defined in Proposition \ref{center}. Hence $\leq$ is a threshold order. But this does not imply that every semiorder on $\RR$ has this form, because the semiorder depends upon the compatible total  order on the reals and there are plenty of those  (in fact there are $2^{\mathfrak c}$ many, where ${\mathfrak c}=2^{\aleph_0}$). For example, viewing $\RR$ as a vector space over the field $\QQ$ of rational numbers, let $H$ be a  basis of $\RR$ (the existence of such  a base, called a Hamel base, requires some form of  the axiom of choice) and for every real $x$ let $supp(x)$, the \emph{support} of $x$, be the set of members of $B$ which appear in a decomposition of $x$ over $\QQ$. Put   a  total order $\preceq$ on $H$ and extend it to a total order on the reals by defining the positive cone $C_H:= \{0\}\cup \{ x\in \RR\setminus \{0\}: lead(x)>0\}$, where $lead(x)$, the \emph{leading coefficient} of $x$, is the coefficient of the largest element of $supp(x)$ which appears in the decomposition of $x$.  According to Proposition \ref{prop:thresholdgroup},  any final segment in the positive cone will lead to a semiorder and, in some instances,  to a threshold order. Some continuity or density conditions on a compatible  total order on the reals  in order to recover the ordinary order are necessary.
This  explains somehow why the representation of infinite semiorders in the reals was a difficult problem \cite{candeal1, candeal}.

There are $2^{\aleph_0}$ compatible total orders on  the additive group $\ZZ^2$. Here is an example of a threshold group on $\ZZ^2$. The positive cone is the set $C:= \{(0,0)\} \cup \{(n,m): n\geq 0, m\geq 1, \; \text{or}\;  n<0\;  \text{and}\;  m\geq 2\}$. The idea is the same as above: set $B:= \{(1,0), (0,1)\}$ then every element of $\ZZ$ is a combination of members of $B$ with integer coefficients. Set $(1,0)\prec (0,1)$; the positive  cone of the extension $\preceq$  is $C_B:= \{0\} \cup \{(n,m): m>0 \text{ or }  m=0 \text{ and } n>0\}$ and the order is a sum of copies of $\ZZ$ indexed by $\ZZ$. Set $\alpha:= (1,1)$  and  $\preceq_{\alpha}$ be the corresponding threshold order. See Figure \ref{grid}.\\

For nonabelian groups this conclusion of Corollary \ref{cor:threshold-abelian} is no longer valid. In \cite{clifford51}, Clifford exhibits an example of a nonabelian and nonfinitely generated totally ordered order-simple group (an ordered group is \emph{order-simple} if every normal and convex subgroup is trivial). We prove that in Clifford's example the set of positive elements contains no proper normal final segment (see Proposition \ref{clifford} in Subsection \ref{subsection-clifford}). Hence, the only compatible semiorder on this group is either the total order or the equality relation.

According to $(b)$ of Proposition \ref{prop:thresholdgroup}, if $G$ is a threshold group obtained via a total order $\preceq$ and a nonempty normal final segment $F$ of $\{x\in G: 0\prec x\}$, the set $I:=G\setminus (-F\cup F)$ cannot be a subgroup distinct from $\{0\}$.

In Section \ref{section:fingen}, we prove:

\begin{proposition}\label{fingen} Every finitely generated totally ordered group $G$ admits a proper normal final segment $F$ such that $I:=G\setminus (-F\cup F)$ is not a group.
\end{proposition}

Two widely studied groups in Group Theory are the Thompson group $\FF$ and the free group $F(S)$ on a set $S$ of generators.

The group $\FF$ was introduced by Richard Thompson in some unpublished handwritten notes in 1965 \cite{thompson} as a possible counterexample to von Neumann conjecture. The group $\FF$ has a collection of unusual properties which have made it a counterexample to many general conjectures in group theory. The group $\FF$ is not simple but its derived subgroup $[\FF,\FF]$ is and the quotient of $\FF$ by its derived subgroup is the free abelian group of rank 2. The Thompson group can be generated by two elements and can be induced with a compatible total order \cite{CannonFloydParry, NavasRivas}.

The free group $F(S)$ is a universal group: Every group is isomorphic to a quotient group of some free group. The free group with finitely many generators can also be induced with a compatible total order \cite{magnus}.

A consequence of Proposition \ref{fingen} is this.

\begin{corollary}\label{free-thompson} The Thompson group $\FF$ and the free group on finitely many generators can be equipped with a compatible semiorder which is not a weak order.
\end{corollary}

\begin{problem} Is it true that on every finitely generated totally ordered group there is a compatible threshold order which is not total?
\end{problem}

\subsection{Examples of semiordered groups and threshold groups}

\begin{example}\label{example1}
\emph{A compatible order $\leq$ on $\ZZ$, the additive group of integers, is a semiorder if and only if $C^+_*:=\{x : x>0\}$ is either empty, that is, all elements of $\ZZ$ are pairwise incomparable, or is of the form $a+\NN$ or $-(a+\NN)$ for some $a\geq 1$. If the cone is nonempty the ordered group is a threshold group}.
\begin{proof}Let $\leq$ be a compatible order on $\ZZ$. Then $\leq$ is a suborder of the natural order or of its dual. Without loss of generality, we may suppose that this is a suborder of the natural order, that is, $\varnothing\neq C^+ \subseteq \NN$. Set $a=\min C^+_*$.
\end{proof}
\end{example}

\begin{example}\label{example3} \emph{Let $G$ be an additive subgroup of the reals and let $\alpha \in \RR^+$. Then the ordered groups $G_{\alpha}$ and $G_{\check{\alpha}}$ defined as in Proposition \ref{center} are threshold groups (note that if $\alpha\not \in G$, the two orders coincide).}
\end{example}

In the next examples, we will define the strict order on the groups rather than the order.

\begin{example}\label{example4}\emph{Let $G=(X,+,\leq)$ be a threshold group and let $K$ be any group. The direct product $K\times G$ ordered by \[(a,b)<(a',b') \mbox{ if and only if } b<b'. \] is a semiordered group}.
\begin{proof}Note that for every $b\in X$, the set $K\times \{b\}$ is an antichain and an autonomous set. So that the order on $K\times G$ is the lexicographical sum over $G$ of copies of $K$. Since the order on $G$ is a semiorder so is the order on $K\times G$.
\end{proof}
\end{example}

\begin{example}\label{example5}\emph{If $C$ is a totally ordered group and $G$ is a threshold group, then the direct product $G\times C$ ordered by
\[(a,b)<(a',b') \mbox{ if and only if }\;  b<b' \mbox{ or } (b=b' \mbox{ and } a< a' \; \mbox{in } G) \; \]
is a threshold group}.
\begin{proof}Note that the order on $G\times C$ is a lexicographical sum over $C$ of copies of $G$. The relation $\leq_{pred}$ on $G\times C$ is the lexicographical sum over $C$ of the relation $\leq_{pred}$ on $G$. Since $G$ is a threshold ordered group its corresponding $\leq_{pred}$ is a total order. Since $C$ is totally ordered it follows that the relation $\leq_{pred}$ on $G\times C$ is a total order and hence $G\times C$ is a threshold group.
\end{proof}
\end{example}

\begin{example}\label{example6} \emph{Let $A$ be a totally ordered group, $F$ be a normal nonempty final segment containing possibly $0$ and $G$ be a  threshold group with  attained threshold $\alpha>0$. On the direct product $A\times G$ we define the following relation:
\[(a,b)<_{(F, \alpha)} (a',b') \; \mbox{ iff}\;   b+\alpha < b' \; \mbox{ in}\; G\; \mbox{or}\; (b'= b+\alpha \; \mbox {and} \; a'-a\in F). \]}
\emph {This relation is an order, it defines a compatible semiorder on $A\times G$. If $F$ is not the union of cosets of some nontrivial convex subgroup of $A$ this is a threshold order; in this case, $\leq_{pred}$ on $A\times G$ is the lexicographical sum of the total order on $A$ indexed by the total order $\leq_{pred}$ on $G$, hence is a total order. We denote by $A \odot_{F, \alpha} G$ the direct product $A\times G$ of these two groups, with the order $\leq _{F, \alpha}$}.
\begin{proof}One can easily verify that $\leq _{F, \alpha}$ is antisymmetric (follows from $\alpha>0$) and transitive. The compatibility of $<_{(F, \alpha)}$ follows from the normality of $F$ and our assumption that $G$ is a threshold group. From $\alpha> 0$ we deduce that for every $b\in G$, the order $<_{(F, \alpha)}$ is the equality on $A\times \{b\}$, that is $A\times \{b\}$ is an antichain. Let $b,b'\in G$ be distinct. If $b'-b>\alpha$, then $(a,b)\leq _{F, \alpha} (a',b')$ for all $a,a'\in A$. Else if $b'-b=\alpha$, then the restriction of $\leq _{F, \alpha}$ to $A\times \{b\}\cup A\times \{b'\}$ is bipartite and $(a,b)\leq _{F, \alpha} (a',b')$ if $a'-a\in F$. We claim that
$(a,b){\leq_{(F, \alpha)}}_{pred} (a',b')$ whenever $b+\alpha \leq b'$ or $a\leq a'$ in $A$ and $b=b'$, which is enough to prove that the quasi-order ${\leq_{(F, \alpha)}}_{pred}$  is total, that is, $<_{(F, \alpha)}$ is a (strict) semiorder. Indeed, let $(a'',b'')$ be such that $(a'',b'') <_{(F, \alpha)} (a,b)$. Then $b''+\alpha \leq b$ and hence $b''+\alpha \leq b''+2\alpha \leq b+\alpha\leq b'$. Therefore, if $b''+\alpha < b$ or $b+\alpha < b'$, then $(a'',b'') <_{(F, \alpha)} (a',b')$. Next we assume that $b''+\alpha = b$ and $b+\alpha = b'$. This in turn leads to $b''+2 \alpha=b'$ and in particular $b''+\alpha<b'$ proving again that $(a'',b'') <_{(F, \alpha)} (a',b')$. We now assume that $b=b'$. Clearly if $b''+\alpha<b$, then $(a'',b'') <_{(F, \alpha)} (a',b')$. Else $a-a''\in F$. Now $a'-a''=a'-a+a-a''$ and since $a<a'$ in $A$ by assumption we infer that $a'-a''\in F$ proving that $(a'',b'') <_{(F, \alpha)} (a',b')$.
\end{proof}
\end{example}

\begin{example}\label{example7}\emph{Let $K, A,B,C$ be four groups and $G:= K\times A\times B\times C$ be their direct product. We suppose that $B$ is an additive subgroup of the reals with an attained threshold $\alpha$ as in Example \ref{example3}. We suppose that $A$ and $C$ are totally ordered groups and that $F$ is a nonempty normal final segment of $A$ which is not an union of cosets of some nontrivial convex subgroup of $A$. The order on $B\times C$ is the lexicographical sum over $C$ of copies of the order on $B$. Let $A':=A\odot_F B$ ordered as in Example \ref{example6}. Finally, the order on $K\times A'\times C$ is the lexicographical sum over $C$ of copies of the order on $K\times A'$. Then $A'\times C$ is a threshold group and $G$ is a semiordered group.}
\begin{proof}The order on $G$ is the lexicographical sum over $C$ of copies of the order on $K\times A'$. Hence, in order to prove that the order on $G$ is a semiorder it is enough to prove that the order on $K\times A'$ is a semiorder. Since the order on $K\times A'$ is the lexicographical sum over $A'$ of copies of $K$ and $A'$ is a semiorder we infer that the order on $K\times A'$ is a semiorder. This proves that $G$ is a semiordered group. The fact that $A'\times C$ is a threshold group follows from $A'$ is a threshold group (see example \ref{example6}).
\end{proof}
\end{example}

\subsection{Decomposition of semiordered groups}
In this subsection,  we prove that the examples given above describe the possible  semiorders on  groups, e.g. Theorem \ref{thm:4} shows that the order on a group is a  semiorder iff it can be expressed as a product like in Example \ref{example7}. Of course, the group does not need to decompose as in this example (this depends upon the existence of  direct factors for the subgroups).

For every ordered group $G= (X, +, \leq) $ there are two particular normal convex subgroups that play an important role in the study of its structure. These are $K(G)$, defined in (\ref{KdeG}) of Subsection \ref{subsection:threshold},  and $I(G)(0)$.

To the poset $(X, \leq)$, we associate its  incomparability graph $(X, \nsim)$.
 As any graph, it  decomposes into connected components. We denote by  $I(G)(0)$ the connected component of $0$. This is a convex normal subgroup of $G$.

If the order is a semiorder, another subgroup, $A(G)$,  enters into the picture.

The set
$$A(G):= \{x\in G: \ZZ\, x \; \text{is an antichain in }\; G\}; $$
is also a normal subgroup of $G$.

In Subsection \ref{pfthmthree groups} we prove:

\begin{theorem}\label{thm:three groups}Let $G:= (X, +, \leq)$ be an ordered group. Then:
\begin{enumerate}[{(a)}]
\item $K(G)$ is the largest autonomous subset in $(X,\leq)$ which contains $0$ and is an antichain;
\item $I(G)(0)$ is the least subgroup of $G$ which contains $H=inc(0)\cup \{0\}$; the quotient group $G/I(G)(0)$ is totally ordered and the order on $G$ is isomorphic to the lexicographical sum of the order of $I(G)(0)$ indexed by the chain $G/I(G)(0)$;
\item If the order is a semiorder, $A(G)$ is a convex normal subgroup and is the largest subgroup of $G$ which is an antichain.
\end{enumerate}
\end{theorem}

We should mention that $K(G)\subseteq A(G)\subseteq I(G)(0)$. Moreover, if $G$ has no torsion element, then $A(G)$ is an infinite antichain or is $\{0\}$. Furthermore, all these sets are convex subsets (this is obvious for $K(G)$ an $A(G)$ since they are antichains, for the convexity of $I(G)(0)$ see Lemma \ref{lem5:4}) with respect to the order on $G$ and hence the quotient groups $G/K(G)$, $G/A(G)$, $G/I(G)(0)$  can be equipped with a partial order (see Lemma \ref{convex}). In fact, the groups $G/A(G)$, $G/I(G)(0)$ are totally ordered. Moreover, the order on $G$ is the lexicographical sum of the cosets of $I(G)(0)$ indexed by the chain $G/I(G)(0)$.  The order structure of $I(G)(0)$ is rather simple if $\leq_{pred}$ is a total order which is not dense. Indeed, in this case the set of positive elements (with respect to $\leq_{pred}$) has a smallest element $a$ and hence $I(G)(0)= \ZZ a$ and there is  some threshold $\alpha \geq a$.

The groups $K(G)$, $A(G)$ and $I(G)(0)$ play an important role in the study of the structure of partially ordered groups whose order is a semiorder. Indeed, it  follows from Theorem \ref{thm:4} below that if the order on a group $G$ is a semiorder and if the groups $K(G)$, $A(G)$ and $I(G)(0)$ are direct factors in $A(G)$, $I(G)(0)$ and $G$ respectively, then $G$ has a decomposition as in Example \ref{example7} above. We note that in Example \ref{example7}, $K(G)= K\times\{(0,0,0)\}$, $A(G)= K\times A\times \{(0,0)\}$, $I(G)(0)= K\times A\times B\times \{0\}$.

The next three theorems give insights on the structure of ordered groups whose order is a semiorder.

\begin{theorem}\label{thm:3} The following properties for an ordered group $G$ are equivalent.

 \begin{enumerate}[(i)]
\item The order on $G$ is a semiorder;
\item $I(G)(0)/K(G)$ is a threshold group;
\item $G/K(G)$ is a threshold group.
\end{enumerate}
\end{theorem}

\begin{theorem}\label{thm:4} Let $G$ be an ordered group equipped with a semiorder. If $G= I(G)(0)$ then
\begin{enumerate}[(a)]
\item The group $G/A(G)$ is an additive subgroup of the real numbers. The image of the order on $G$ by the canonical mapping from
the group $G$ onto its quotient group $G/A(G)$ is a threshold order and if $A(G)\not =\{0\}$ the threshold is attained.
\item The order on $G$ is the lexicographical sum of copies of  $K(G)$ indexed by $G/K(G)$.

\item If $K(G)=\{0\}$, $A(G)\not = \{0\}$ and $A(G)$ is a direct factor of $G$ then the order on
$G$ is the order of $A(G)\odot_F(G/ A(G))$ for some final segment $F$ of $A(G)$.

\end{enumerate}
\end{theorem}

\begin{theorem}\label{thm:6}Let $G$ be an ordered group. If the order is a semiorder, but not an antichain,  and $G=I(G)(0)$, then the following propositions are true:
\begin{enumerate}
\item Every maximal chain of  $G$ is isomorphic to the chain of integers.
\item $K(G)=\{0\}$ and $G$ has no infinite antichain if and only if $G$ is isomorphic to the group of integers equipped with a threshold order.
\end{enumerate}
\end{theorem}

The following result gives insights on the structure of convex subgroups of an ordered group whose order is a semiorder.

\begin{theorem}\label{convexsubgroup}
Let $G=(X,+,\leq)$ be an ordered group such that $\leq$ is a semiorder. If $H$ is a convex subgroup of $G$, then either $H\subseteq A(G)$, in which case $H$ is an antichain, or $I(G)(0)\subseteq H$. If the latter holds and $H$ is a normal subgroup of $G$, then $G/H$ is totally ordered and $G$ is the lexicographical sum of the cosets of $H$.
\end{theorem}

A consequence of Theorems \ref{thm:three groups} and \ref{thm:4} is the following proposition.

\begin{proposition}\label{prop:existence}
For a group $G$ the following properties are equivalent:

\begin{enumerate}[{(i)}]
\item $G$ can be endowed with a compatible semiorder distinct from the equality relation;
\item $G$ can be endowed with a compatible weak order distinct from the equality relation;
\item $G$ has a proper normal subgroup $H$ whose quotient $G/H$ is totally orderable.
\end{enumerate}
\end{proposition}
\begin{proof}
$(i)\Rightarrow (iii)$ Suppose that $G$ is endowed with a compatible semiorder. We consider two cases.\\
\textbf{Case 1.}  $I(G)(0)\not= G$. According to $(b)$ of Theorem \ref{thm:three groups}, the group $H= I(G)(0)$ is a normal subgroup of $G$ and the quotient $G/H$ is totally ordered.\\
\textbf{Case 2.} $I(G)(0)\not= G$. According to $(a)$ of Theorem \ref{thm:4}, the group $A(G)$ is a normal subgroup of $G$ and the quotient $G/A(G)$ is a subgroup of the additive group of the set $\RR$ of  reals numbers, hence is totally orderable.\\
$(iii)\Rightarrow (ii)$ This is Corollary \ref{cor:existenceweak}.\\
$(ii)\Rightarrow (i)$ Obvious: a weak order is a semiorder.
\end{proof}

This paper contains six  more sections. Section \ref{subsection:convex} contains some general properties of  convex sets and autonomous sets of a quasi-ordered group. Section \ref{section:basic} contains   properties of  normal subsets of a group and the proofs of Propositions  \ref{claim:weakorder},  \ref{prop:thresholdgroup}  and Theorem \ref{claim:interval}. Section \ref{section:properties of the groups} describes some properties of the groups $I(G)(0)$ and $A(G)$ and contains the proofs of  Theorems \ref{thm:three groups} to \ref{convexsubgroup}  about $K(G)$, $I(G)(0)$ and $A(G)$.
Section \ref{subsection-clifford} contains a description of Clifford's example and a proof of Theorem \ref{fingen}.  Section \ref{subsection:dimension} contains an effective proof that a threshold group with attained threshold has dimension at most three.
Section \ref{section:comparison} introduces an  extension of the quasi-order of embeddability between posets and describes some equivalence classes.

\section{Some general properties of convex sets and autonomous sets of a quasi-ordered group}\label{subsection:convex}
\begin{lemma}\label{lem:interval} If $G:= (X, +, \leq)$ is a quasi-ordered group and $\overline 0$ be the equivalence class of $0$ with  respect to the quasi-order $\leq$. Then $\overline 0$ is the least convex subgroup of $G$. The subgroup  $K(G)$ is an autonomous subset in $(X, \leq)$ and the  direct sum of the cosets of $\overline{0}$ it contains. If $\overline{0}=\{0\}$, that is, $\leq$ is an order, then $K(G)$ is the largest autonomous subset in $(X,\leq)$ which contains $\{0\}$ and is an antichain. Furthermore $K(G)\setminus \{0\}$ is the set of isolated vertices of the poset $inc(0)$.
\end{lemma}
\begin{proof}
The first statement is obvious. For the second, let
$inc(0)$ be the subset of $X$ of elements incomparable to $0$. We notice at first that $\overline{0}\subseteq K(G) \subseteq inc(0) \cup \overline 0$. The first inclusion is obvious; for the second, let $x\in K(G)$ that is $x\equiv_{pred} 0$. If $x\not \in inc(0)\cup \overline 0$ then either $x<0$ or $x>0$. If $x<0$ then from $0\leq_{pred} x$ we get $x<x$ which is impossible; if $0<x$ then since  $x\leq_{pred} 0$ we get  $0<0$ which is also impossible. We also note that $\overline{0}$ is an autonomous subset of $(X, \leq)$ and hence is an autonomous subset in $K(G)$. Next we observe the following. Let $A$ be an autonomous subset containing $0$. If the quasi-order on $A$ is either the equality relation or the complete relation then $A \subseteq K(G)$. Indeed, let $a\in A$. We claim that $a\equiv_{pred} 0$ hence $a\in K(G)$. Indeed, let $x<a$. From our assumption on $A$ we deduce that $x\not \in A$. Since $A$ is autonomous, $x<0$ hence $a\leq_{pred} 0$. Similarly, $0\leq _{pred} a$ hence $a\equiv_{pred} 0$. Next we prove that $K(G)$ is an autonomous subset in $(X, \leq)$. Indeed, let $x\not \in K(G)$ and $y,y'\in K(G)$. If $x\leq y$, then $y\not \leq x$, otherwise $x\equiv y$ and $x\in \overline{0} \subseteq K(G)$ which is excluded. Since $y\equiv_{pred} y'$ we have $x<y'$, hence $x\leq y'$. Similarly, since $\leq_{pred}= \leq_{succ}$, then $y\leq x$ is equivalent to $y'\leq x$. Hence $K(G)$ is autonomous.
Suppose that $\overline{0}= \{0\}$. In this case, $K(G)$ is an antichain. From our observation, this is the largest autonomous subset containing $0$ which is an antichain. Now, let $x\in K(G)\setminus \{0\}$, then $x\in inc(0)$. Suppose that $x$ is not isolated in $inc(0)$, that is either $x<y$ or $y<x$ for some $y\in inc(0)$. In the first case, from $0\leq_{succ}x$ and $x<y$ we obtain $0<y$ contradicting $y\in inc(0)$. Using the fact that $x\leq_{prec}0$, we get the same conclusion in the second case. Hence $x$ is isolated. Conversely, let $x$ be isolated in $inc(0)$. If $0\not \leq_{succ}x$ then there is some $y\in X$ such that $x<y$ but $0\not < y$. Necessarily, $y\in inc(0)$, hence $x$ is not isolated, a contradiction. Hence $0\leq_{succ} x$. Similarly, $x\leq_{pred} 0$. Since $\leq_{pred}= \leq_{succ}$, $x\equiv_{pred}0$, and therefore $x\in K(G)$.
\end{proof}

Let $G$ be a group, $H$ be a normal subgroup of $G$ and $G/H$ be the quotient group. Suppose that $G$ is equipped with a compatible quasi-order $\leq$. Let $\leq_H$ be the image of $\leq$ by the quotient map. That is $\alpha\leq_H\beta$ in $G/H$ if  $a\leq b$ for some $a\in \alpha$ and $b\in \beta$ (equivalently,  $a\leq b$ for every $a\in \alpha$ and some $b\in \beta$). The relation $\leq_H$ is a compatible quasi-order (Proposition 4, page 25 of \cite{kopytov}). Furthermore:

\begin{lemma}\label{convex} Let $G:=(X,+,\leq)$ be a quasi-ordered group and  $H$ be a normal subgroup of $(X,+)$.  Then $\leq_H$ is a compatible order relation on $G/H$ if and only if $H$ is a convex subgroup of $G$. If $H$ is a convex and autonomous subset in $(X, \leq)$, then the quasi-order on $G$ is the lexicographical sum of copies of $H$ indexed by $G/H$. Furthermore, $\leq_{pred H}$, the image of $\leq_{pred}$, is included into $\leq_{H pred}$, the quasi-order associated to $\leq_H$. If $H= K(G)$ these quasi-orders coincide.
\end{lemma}
\begin{proof}
Suppose that $\leq_H$ is a compatible order. Let $x,y,z$ with $x,y\in H$ and $x\leq z\leq y$. In $G/H$ we have $0\leq_H z+H\leq_H 0$. Since $\leq_H$ is an order, $z+H=0$ amounting to $z\in H$. Hence, $H$ is convex. For the converse, note first that from (a) of Lemma \ref{pred=suc},  $\overline {0}$, the equivalence class of $0$ w.r.t.  to $\leq$,  is a normal  subgroup of $H$.  Thus, if $H$ is convex it contains $\overline {0}$ (indeed, let $x\in \overline {0}$. We have $0\leq x\leq 0$. Since $0\in H$ and $H$ is convex, $x\in H$). More generally, let $\alpha\in G/H$. Suppose that for $0\in G/H$ we have $0\leq_H \alpha \leq_ H 0$. Then, there are $a\in \alpha$ and $b\in 0=H$ such that $0\leq a\leq b$. Since $H$ is convex, $a\in H$ hence $\alpha=0$ as required. If $H$ is convex and autonomous, then being an autonomous subset in $(X,\leq)$, each coset is also an autonomous subset and is order-isomorphic to $H$, hence $(X,\leq)$ is isomorphic to the lexicographical sum of $H$ indexed by $G/H$.

Next, we prove that the quasi-order $\leq_{pred H}$ is included into  $\leq_{H pred}$. Let $\alpha, \beta \in G/H$; we set $\alpha <_H \beta$ if $\alpha \leq_H \beta$ and $\alpha\not\leq_H \beta$. Let $\alpha\leq _{predH}\beta$ and $\gamma <_H\alpha$. There are $c\in \gamma$ and $a\in \alpha$ such that $c<a$; also there is $b\in \beta$ such that $a\leq_{pred} b$. Since $a\leq_{pred}b$ we have $c<b$. We claim that $\gamma\not = \beta$ from which follows $\gamma <_H \beta$ and $\alpha\leq _{Hpred}\beta$. Indeed, otherwise $\gamma= \beta$. Since $c< a$ and each coset is an interval, we have $b<a$. Since $a\leq_{pred}b$ this yields $b<b$ which is impossible.

Now, suppose that $H=K(G)$. We claim that $H$ is convex and autonomous subset. According to Lemma \ref{lem:interval} it is autonomous. We show that it is convex. Let $y\leq x\leq 0$ with $y\in H$. If $0\leq x$ then $0\equiv x$ and therefore $x\in H$. Otherwise $x<0$. Since $y\in H$, $0\equiv_{pred} y$ hence $x< y$, contradicting $y\leq x$. Now, let $\alpha \in G/H$ be such that $0\leq_{H pred} \alpha$. We claim that $0\leq_{pred H}\alpha$. If $0=\alpha$ there is nothing to prove. Suppose $\alpha\not =0$. Pick $a\in \alpha$. We prove that $0\leq _{pred} a$ from which the claim follows. Let $x<0$. Since $H=K(G)$, $x\not \in H$ by Lemma \ref{lem:interval}. Hence, $x+H<0$. Since $0\leq_{H pred}\alpha$ we have $x+H<_{H} \alpha$. Hence $x<_Ha'$ for some $a'\in \alpha$. Since $H=K(G)$ each coset is an interval, hence $x<a$, proving that $0\leq _{pred} a$.
\end{proof}

\noindent \rem Note that in general $\leq_{H pred}$ is not included into $\leq_{H pred}$. For an example, let $G:=\ZZ \times \ZZ/2\ZZ$ be the direct product of the totally ordered group $\ZZ$ with $\ZZ/2\ZZ$ ordered by the equality relation. Let $H:= \ZZ \times \{0\}$. Then $\leq_H$ is the equality relation on $G/H$, hence $\leq_{H pred}$ is the complete relation on the quotient. On the other hand, $\leq_{pred}= \leq$ and $\leq_{predH}$ is the equality relation on the quotient.

A set of subsets that forms a totally ordered set under inclusion will be called a chain of subsets. For the following two lemmas we refer to  \cite{glassbook} Chapter 3 (Lemmas 3.1.1 and 3.1.2).

\begin{lemma}\label{unionconvex} If $G$ is a partially ordered group and $\{C_i : i \in I\}$ is a family of convex subgroups of $G$, then $\cap\{C_i : i \in I\}$ is a convex subgroup of G. If $\{C_i : i \in I\}$ is a chain of convex subgroups of $G$, then $\cup\{C_i : i \in I\}$ is a convex subgroup of $G$.
\end{lemma}

Let $G$ be a partially ordered group and $g \in G\setminus \{0\}$. Let $\mathcal{V}$ be the set of all convex subgroups $V$ of $G$ such that $g \not \in V$. By Lemma \ref{unionconvex} the union of any chain of subsets in $\mathcal{V}$ is itself in $\mathcal{V}$. Since $\{0\} \in \mathcal{V}$, we have maximal elements in $\mathcal{V}$ by Zorn's Lemma. Hence for each  $g\in G\setminus\{0\}$, there is a convex subgroup $C_g$ maximal not containing $g$; i.e.: the union of all convex subgroups not containing $g$. If $G$ is totally ordered, Zorn's Lemma is not required.

\begin{lemma}\label {largestconvex}
 If $G$ is a totally ordered group, the set of all convex subgroups of $G$ is a complete chain. Hence,  for every $g \in G\setminus \{0\}$ there is a largest convex subgroup  $C_g$ of $G$ which does not contain $g$.
\end{lemma}

Let us recall that an ordered group is \emph{archimedean} if for all $x, y>0$ there is a nonnegative integer $n$ such that
$n x\not \leq y$. Note that for a totally ordered group this means $n x>y$. The  following characterisation of totally ordered archimedean groups is known as H\"older's theorem  \cite{holder}. For completeness, we indicate the scheme of the proof.
\begin{theorem}\label{archimedean} Let $G$ be a totally ordered group. The following propositions are equivalent.
\begin{enumerate}[(i)]
 \item $G$ is archimedean.
 \item $G$ is isomorphic as an ordered group to a subgroup of the additive group of the real numbers (in particular, $G$ is abelian).
 \item $G$ has no nontrivial convex subgroups.
\end{enumerate}
\end{theorem}
\begin{proof}
The implication $(i) \Rightarrow (ii)$ is due to H\"older, see \cite{glassbook} Theorem 4.a, page 56.\\
$(ii) \Rightarrow (iii)$ It is well known that every subgroup of $\RR$ is either discrete, that is it is of the form $x\ZZ$ for some $x\in \RR$, or is dense in $\RR$. Hence a nontrivial subgroup of $\RR$ cannot be convex.\\
$(iii) \Rightarrow (i)$ Suppose $G$ has no nontrivial convex subgroups and assume for a contradiction that it is not archimedean. There exists then $x,y\in X$ such that $n x<y$ for all $n\in \ZZ$. Let $H$ be the smallest interval of $G$ containing $x\ZZ$. Then $H$ is by definition convex. Moreover $H\neq \{0\}$ and $H\subseteq \{z\in X : -y\preceq z\preceq y\}$ since $n x<y$ for all $n\in \ZZ$ and therefore $H\neq G$. We now prove that $H$ is a subgroup contradicting our assumption. Clearly $0\in H$. Let $h,h'\in H$. By definition of $H$ there exist $n,m,n',m'\in \ZZ$ such that $nx<h<mx$ and $n'x<h'<m'x$. Therefore $(n-m')x<h-h'<(m-n')x$ proving that $h-h'\in H$ and hence $H$ is a subgroup of $G$. The proof of the theorem is now complete.
\end{proof}

\section{Normal subsets and proofs of Propositions  \ref{claim:weakorder}, \ref{prop:thresholdgroup} and Theorem \ref{claim:interval}}\label{section:basic}

\subsection{Normal subsets}\label{normal}
Let $G:= (X, +)$ be a group and let $A$ be a subset of $G$. For $u\in G$ define $u+A:= \{u+x: x\in A\}$. Similarly define $A+u$. We say that $A$ is \emph{normal} if for all $a\in A$ and $u\in G$ we have $-u+ a+ u\in A$. Incidently, $G$, $\{0\}$ and $\varnothing$ are normal.\\
\textbf{Fact 1:} $A$ is normal iff $-u+A+u=A$ for all $u\in G$.\\
 \emph{Proof of Fact 1:} Fact 1 is a well known, elementary fact of beginning group theory. We do not include its proof. \hfill $\Box$

We give below some relevant examples of normal sets.
\begin{enumerate}
 \item If $A$ is a singleton, say $A=\{a\}$, then $A$ is normal iff $a\in Z(G)$.\\

\noindent Now suppose $G$ equipped with a compatible partial order and $G^{+}:= \{x\in G: x\geq 0\}$. Then, as it is immediate to see:
\item $G^{+}$ is normal. \\
\noindent Next, let $\uparrow a$ be the final segment generated by $a$, that is, $\uparrow a:= \{x\in G: a\leq x\}$. Then, trivially, $\uparrow a= a+G^{+}=G^{+}+a$.

\item  $\uparrow a$ is normal iff $a\in Z(G)$. If $G$ is totally ordered, it suffices that $a$ commutes with all the elements of $G^{+}$. \\
     Indeed, suppose $\uparrow a$ is normal. Since $a\in \uparrow a$ and $\uparrow a$ is normal we have $-u+a+u\in \uparrow a$  that is $-u+a+u\geq a$ for every $u\in G$. Let $v=-u$, we have $-v+a-v\geq a$, that is $u+a-u\geq a$ and equivalently $a\geq -u+a+u$. Hence $-u+a+u= a$ and therefore $a$ and $u$ commute. Reciprocally, suppose that $a$ and $u$ commute for all $u$ in $G$. Then
$-u+ \uparrow a+ u= -u+ a+G^{+}+u=a-u+G^{+}+u= a+G^{+}=\uparrow a$ and we are done.

If $a$ commutes with an element $u$ it commutes with $-u$. If $G$ is totally ordered, then every nonzero element is either positive or negative. Hence, if $a$ commutes with all elements of $G^{+}$ it commutes with all elements of $G^{-}:=-G^{+}$, hence with all elements of $G$.

\item Let $inc(0)$ be the set of elements incomparable to $0$. Then $inc(0)$, $I:=inc(0) \cup \{0\}$ and $K(G)$ are convex subsets of $(X, \leq)$ and normal subsets of $G$. Furthermore, $I$ is a convex subset of $(X, \leq_{pred})$.

The first assertion is obvious. For the second, let $u\leq_{pred} g \leq_{pred} v$ with $u, v\in I$. We prove that $g\in I$. Suppose that $g\not\in I$. Since $K(G)\subseteq I$,  $g\not \in K(G)$ and either $0< g$ or $g< 0$. Suppose $0<g$ then from $g\leq_{pred} v$ we get $0<v$ and hence $v\not \equiv_{pred} 0$. Thus $v\not \in I$, a contradiction. If $g<0$, then using the fact that $\leq_{succ}= \leq_{pred}$ we get a contradiction too.
\end{enumerate}


\begin{lemma}\label{lem:semiorder}Let $\preceq$ be a compatible total quasi-order on a group $G=(X, +)$ and let $F$ be a final segment of $(X, \preceq)$ not containing $0$. Define a binary relation $\leq$ on $G$ as follows:
\[x\leq y \mbox{ if } x=y \mbox{ or } y-x\in F.\] Then:
\begin{enumerate}[(a)]
\item $\leq$ is an order.
\item $\preceq$ is an extension of $\leq$ and both $\leq_{pred}$ and $\leq_{succ}$ extend  $\preceq$.
\item $\leq$ is a semiorder.
\item $\leq$ is compatible iff $F$ is normal.

 \end{enumerate}
\end{lemma}
\begin{proof} We observe that since $\preceq $ is total, $F$ is a final segment of $G^+:=\{x\in X: 0 \preceq x\}$  equipped with the compatible quasi-order $\preceq$. Hence, $F$ is a cone, i.e. $x,y\in F$ implies $x+y\in F$. Indeed, if $x\in F$ and $y\in F$ then since $0\preceq y$ and $\preceq$ is a compatible quasi-order, we have $x\preceq x+y$. Since $x\in F$ and $F$ is a final segment $x+y\in F$.

\noindent $(a)$ The fact that $\leq$ is reflexive follows from the definition. Next,
 $\leq$ is transitive. Indeed, let $x,y,z$ be such that $x\leq y \leq z$. If $x= y$ or $y=z$ then $x\leq z$. If not, then $y-x\in F$ and $z-y\in F$. Since $F$ is a cone, $z-x=z-y+y-x\in F$ hence $x\leq z$. Finally, $\leq$ is antisymmetric, indeed if $y-x\in F$ and $x-y\in F$, then since $F$ is a cone $0\in F$ which is impossible.

\noindent $(b)$ Let $x\leq y$. If $x=y$ then trivially, $x\preceq y$. Thus we may suppose $x<y$. In this case, $y-x\in F$ by definition of $\leq$. Since $0\prec y-x$ and $\preceq$ is compatible we have $x\prec y$, proving that $\preceq$ is an extension of $\leq$. Next we prove that $\leq_{pred}$ extends $\preceq$, that is, $x\preceq y$ implies $x\leq_{pred}y$. Let $z<x\preceq y$. Our aim is to show that $z<y$, that is, $y-z\in F$. As $z<x\preceq y$ and $\preceq$ is compatible we have $x-z\in F$ and $y-x\succeq 0$ furthermore $x-z\preceq y-x+x-z= y-z$. Since $F$ is a final segment $y-z\in F$ as required. Similarly, suppose that $y<z$. We have $z-y\in F$ and since $x\prec y$, $y-x\succeq 0$ hence $z-y\preceq z-y + y-x=z-x$. Since $F$ is a final segment $z-x\in F$ as required.

\noindent $(c)$ Since $\preceq$ is a total quasi-order and, according to item $(b)$, both $\leq_{pred}$ and $\leq_{succ}$ extend it, the intersection $\leq_{pred}\cap \leq_{succ}$ is a total quasi-order. Hence, according to Lemma \ref{critical}, $\leq$ is a semiorder.

 \noindent $(d)$ Let $x,y$ be such that $x\leq y$. If $z$ is any element of $G$ then $y+z-(x+z)= y-x$, hence $x+z\leq y+z$. On an other hand, $z+y-(z+x)= z+y-x-z$. Hence, if $F$ is normal and $y-x\in F$ then $z+y-(z+x)\in F$, proving that $z+x\leq z+y$. Thus $\leq$ is compatible. Suppose now that $\leq$ is compatible and let $y\in F$ and $x\in G$. From the definition of $\leq$ we have $F=\{t\in X : t>0\}$. Hence, $0<y$ and therefore $x<y+x$ since $\leq$ is compatible. But then $0<-x+y+x$, that is, $-x+y+x\in F$ proving that $F$ is normal.
\end{proof}

%
\begin{lemma}\label{lem:propsemiorder}Let $\preceq$ be a compatible total order on a group $G:=(X, +)$, let $F$ be a normal final segment of this order not containing $0$ and let $\leq_F$ be defined by $x\leq_F y$ if $x=y$ or $y-x\in F$. Then $\leq_F $ is a compatible semiorder and, for $G_F:= (X, +, \leq_F)$, $K(G_F)$ is the largest convex normal subgroup of $(X, +, \preceq)$ such that $I:= X\setminus (-F\cup F)$ is a union of cosets.
\end{lemma}
\begin{proof} The first part follows from $(d)$ of Lemma \ref{lem:semiorder}. Next, $I$ is $inc(0)$ in $G_F$. Since $K(G_F)\setminus \{0\}$ is the set of isolated points of $inc(0)=I$ we infer that for all $x\in inc(0)$, $x+K(G_F) \subseteq I$ and hence $I$ is a union of cosets of $K(G_F)$. Finally, let $K$ be a convex normal subgroup such that $I$ is a union of cosets of $K$. Then every element of $K$ is isolated in $I$ and hence $K \subseteq K(G)$.
\end{proof}

%

\subsection{Proof of Proposition \ref{claim:weakorder}}\label{subsection:claim:weakorder}
\begin{proof}
$(i)\Longleftrightarrow (ii)$ The order on $G$ is a weak order if and only if $1\oplus 2$ does not embed into $G$. Since translations preserve the ordering, this amounts to the fact that the $2$-element chain does not embed into $inc(0)$, that is $inc(0)$ is an antichain.\\
$(ii)\Longleftrightarrow (iii)$ Let $a,b \in inc(0)\cup \{0\}$. If $inc(0)$ is an antichain, then so is $inc(0) \cup \{0\}$, hence $a$ and $b$ are incomparable and therefore $a-b\in inc(0)\cup \{0\}$. This proves that $inc(0)\cup \{0\}$ is a subgroup.\\
Conversely, suppose that $inc(0)\cup \{0\}$ is a subgroup. Let $x,y$ be two distinct elements of $inc(0)\cup \{0\}
$ then $x-y\in inc(0)\cup \{0\}$. Since,  in fact, $x-y\in inc(0)$ it follows that $x$ and $y$ are incomparable.\\
Next we suppose that any of the above conditions (i), (ii) or (iii) hold. Then $I=inc(0)\cup \{0\}$ is a normal subset of $G$, that is, for all $x\in I$ and for all $g\in G$ we have $g+x-g\in I$. Indeed, $x$ is incomparable to $0$ if and only if $g+x$ is incomparable to $g$ if and only if $g+x-g$ is incomparable to $g-g=0$. Hence $I$ is a normal subgroup of $G$. To prove the remaining statement we first prove that $I$ is an autonomous set, that is, every element not in $I$ is either larger than all elements of $I$, or is smaller than all elements of $I$ or is incomparable to all elements of $I$. Let $a,b\in inc(0)\cup \{0\}$ and let $x\not \in inc(0)\cup \{0\}$ be such that $x<a$. Then $x-b<a-b$. Since $inc(0)\cup \{0\}$ is a subgroup we infer that $a-b\in inc(0)\cup \{0\}$ and $x-b\not \in inc(0)\cup \{0\}$, that is, $x-b$ is comparable to $0$, or equivalently, $x$ is comparable to $b$. But then we must have $x<b$ because otherwise $b<x<a$ which is impossible. The required conclusion follows from Lemma \ref{convex}.
\end{proof}

\subsection{Proof of Proposition \ref{prop:thresholdgroup}}\label{sub:threshold}

Assertion $(1)$ is Lemma \ref{lem:propsemiorder}.
Let us prove Assertion $(2)$.
Suppose that $G$ is a threshold group. By definition, $\leq_{pred}$ is a total order. By $(c)$ of Lemma \ref{pred=suc} this total order is compatible. Let $I:=inc(0)\cup \{0\}$ and
$F:=\{x\in G : 0\leq_{pred} x\}\setminus I$. By $(d)$ of Lemma \ref{lem:semiorder} $F$ is a normal final segment; by Item (4) $I$ is a convex subset of $\leq_{pred}$. Hence, $F=\{x\in X: 0< x\}$. It follows that $x<y$ in $G$ if and only if $y-x\in F$ proving $(a)$. Suppose that there is a convex normal subgroup $K$ of $(X, \leq_{pred})$ such that $I$ is a union of cosets of $K$. We prove that $K\subseteq K(G)$, the equivalence class of $0$ with respect to $\leq_{pred}$. Since $G$ is a threshold group, $K(G)=\{0\}$, hence $K=\{0\}$ and $(b)$ holds. So, let $x\in K$. We prove that $x\equiv_{pred}0$, that is $0\leq_{pred}x$ and $x\leq_{pred}0$.  Let $z<0$. Then $z\in -F$. Since from our hypothesis on $K$, $-F$ is a union of cosets of $K$, the coset containing $z$, namely $z+K$, is included into $-F$, hence $z-x\in -F$ that $z<x$. This proves that $0\leq_{pred}x$. Let $z<x$. We have $z-x<0$ that is $z-x\in -F$. Since $z$ belongs to $z-x+K$, the coset containing $z-x$ and $-F$ is a union of cosets, $z\in -F$ hence $z<0$. This proves that $x\leq_{pred}0$.

Conversely, suppose that $\preceq$ is a compatible total order on $G$ and $F$ is a normal final segment of
$\{x\in G: 0\prec x\}$ satisfying conditions $(a)$ and $(b)$ of the proposition. From $(a)$ we deduce that $F=\{x\in X: 0< x\}$ and hence $I=inc(0)\cup \{0\}$. We now prove that $\leq_{pred}$ is an extension of $\preceq$. Indeed, let $x,y$ be such that $x\preceq y$ and let $a<x$. Then $x-a\in F$ and $x-a\preceq y-a$. Hence, $y-a\in F$ and therefore $a<y$. The group $K(G)$, equivalence class of $0$ with respect to $\leq_{pred}$, is a normal subgroup of $G$. Since this is an antichain of $(X, \leq)$ it is contained in $I$. Since $\leq_{pred}$ is an extension of $\preceq$ is a convex subset of $I$ with respect to the order $\preceq$. We claim that $I$ is a union of cosets of $K(G)$. According to condition $(b)$, we will have $K(G)=0$, proving that $\leq_{pred}$ is a total order and hence that $G$ is a threshold group. Our claim amounts to the fact that every coset of $K(G)$ is either a subset of $-F$, or $I$, or $F$. If this fact does not hold, then there exists some coset which meets $-F$ and $I$ or meets $I$ and $F$. Without loss of generality, we may suppose that the latter holds. Then there are $x\in I$ and $y\in F$ such that $x\equiv_{pred}y$. This is impossible, indeed, since $y\in F$ we have $0<y$; since $x\equiv_{pred}y$, it follows $0<x$ which is impossible since $x\not \in F$. \hfill $\Box$

\subsection{Proof of Theorem \ref{claim:interval}}
Let $G:=(X,+,\leq)$ be an ordered group. The fact that $(X,\leq)$ is a semiorder if and only if $inc(0)$ is bipartite follows immediately from the case $n=3$ of Theorem \ref{2+2,3+1} and the fact that translations preserve the order. Indeed, $inc(0)$ is bipartite iff the $3$-element chain $3$ does not embed into $inc(0)$. Since translations on $G$ preserve the order, this latter condition amounts to the fact that $1\oplus 3$ does not embed into $G$. According to the case $n=3$ of Theorem \ref{2+2,3+1}, this is equivalent to $2\oplus 2$ does not embed into $G$. Hence $inc(0)$ is bipartite iff $1\oplus 3$ and $1\oplus 3$ do not embed into $G$, that is the order on $G$ is a semiorder.\\
$(i)\Rightarrow (iii)$ Suppose that $G:=(X,+,\leq)$ is a threshold group. Suppose that $inc(0)$ is not prime. Let $A$ be a nontrivial autonomous subset in $inc(0)$. We claim that if  $x,x'$ are two distinct elements of $A$ then $x\equiv_{pred} x'$ and thus $x-x'\in K(G)$. Since $G$ is a threshold group, $K(G)= \{0\}$, hence $x=x'$, a contradiction. In order to prove this claim, we observe first that since $inc(0)$ is bipartite, the set $inc(0)^{-}:= \{y\in inc(0): y\leq_{pred} 0\}$ is an antichain of $(X, \leq)$ (if this set contains two elements $y'$, $y''$ with $y'<y''$ then, since $y''\leq_{pred}0$, we have $y<0$, which is impossible); similarly, the set $inc(0)^{+}:= \{y\in inc(0): 0\leq_{pred} y\}$ is an antichain. Next, either $inc(0)^{-}$ or $inc(0)^{+}$ contains $A$. Indeed, the poset $inc(0)$ is bipartite and, since $K(G)= \{0\}$, no element of $inc(0)$ is isolated (i.e. every element $x$ of $inc(0)$ is above or below some element of $inc(0)$) (Lemma \ref{lem:interval}). Thus,  the comparability graph of $inc(0)$ is connected (otherwise, it would have at least two connected component; each one being nontrivial, this graph would contain the direct sum of two edges, hence $inc(0)$ would contain $2\oplus 2$ contradicting the fact the order on $G$ is a semiorder). The connectedness of this graph implies that $A$ cannot contain a vertex in each part of the bipartition (if there are $u',u''\in A$ with $u'\leq_{pred} 0\leq_{pred} u''$, then every $v'\in inc(0)^{-}\setminus A$ is incomparable to $u'$ hence incomparable to all the elements of $A$ since $A$ is autonomous; similarly, every $v'\in inc(0)^{-}\setminus A$ is incomparable to all elements of $A$; since $A$ is nontrivial it is distinct of $inc(0)$, hence we may suppose that there is some $v'\in inc(0)^{-}\setminus A$. But there is no path in the comparability graph connecting $v'$ to $u''$). Finally, suppose that $A\subseteq inc(0)^{+}$. Since $\leq_{pred}$ is a total order, with no loss of generality we may suppose $x\leq_{pred} x'$. We claim that $x'\leq_{pred} x$. For that, let $y<x'$. If $y\in inc(0)$ then $y <x$ since $A$ is autonomous; if $y\not \in inc(0)$ then $y<0$. Since $0\leq_{pred} x$, $y<x$. This proves our claim. If $A\subseteq inc(0)^{-}$, we show that $x\equiv_{succ} x'$; since $\equiv_{succ}\, =\, \equiv_{pred}$ we obtain also that $x'\equiv_{pred} x$. Hence, $inc(0)$ is prime.\\
\noindent $(iii) \Rightarrow (ii)$ Obvious.\\
\noindent $(ii)\Rightarrow (i)$ If $inc(0)$ is empty, $G$ is totally ordered, hence it is a threshold order. Suppose that $inc(0)$ is nonempty. Since it is bipartite, the order on $G$ is a semiorder. Since $inc(0)$ has no isolated vertex, Lemma \ref{lem:interval} asserts that  $K(G)= \{0\}$, hence this semiorder is a threshold order.\hfill $\Box$

\section{Some properties of the groups $I(G)(0)$ and $A(G)$ and proofs  of
Theorems~\ref{thm:three groups} to \ref{convexsubgroup}} \label {section:properties of the groups}

The decomposition of the incomparability graph of a poset into connected components is expressed in the following lemma which belongs to the folklore of the theory of ordered sets.
\begin{lemma}\label{interval} If $P:= (X, \leq)$ is a poset, the order on $P$ induces a total order on the set $P/\overline \nsim$ of connected components of $Inc(P)$ and $P$ is the lexicographical sum of these components indexed by the chain $P/\overline \nsim$. In particular, if $\preceq$ is a total order extending the order $\leq$ of $P$, each connected component $A$ of $Inc(P)$ is an interval of the chain $(X, \preceq)$.
\end{lemma}

We denote by $\omega$ the chain of nonnegative integers and by $\omega^*$ the chain of negative integers.

\begin{lemma}\label{lem2:6} If $P$ is a semiorder and $A$ is a connected component of $Inc(P)$ then $\omega+1$ and $1+\omega^*$ do not embed into $(A, \leq)$; equivalently, chains of $(A, \leq)$ embed into $\ZZ$.
\end{lemma}

\begin{proof} Since $P$ is a semiorder, the intersection of $\leq_{pred}$ and $\leq_{succ}$ is a total quasi-order (Lemma \ref{critical}). Let $\preceq$ be a total order included into $\leq_{pred}$ and $\leq_{succ}$. According to implication (\ref{equa:1}) of Subsection \ref{succ and pred}, $\leq_{pred}$ and $\leq_{succ}$ extend $\leq$. Hence $\preceq$ extend $\leq$. From Lemma \ref{interval},
$ A$ is an interval of $(P, \preceq)$.
\begin{claim} \label{claim: length1}
If $x\prec y$ and if $x$ is incomparable to $y$ with respect to $\leq$ then $[x, y]_{\preceq}:=\{z : x\preceq z \preceq y\}$ is an antichain
\end{claim}
\noindent{\bf Proof of Claim \ref{claim: length1}.} Suppose for a contradiction that there are elements $a,b$ such that $x\preceq a<b\preceq y$. Since $a<b$ and $b\leq_{pred}y $ we  have $a <y$. Since $\leq_{pred}= \leq_{succ}$ and $x\leq_{succ} a$ we have $x<y$. A contradiction since $x$ and $y$ are incomparable.\hfill $\Box$

Since $A$ is connected, any two elements $x,y$ of $A$ are joined by a path. We denote by $d_{Inc(P)}(x,y)= n$ the length of the shortest path from $x$ to $y$. From Claim \ref{claim: length1} we obtain:

\begin{claim}\label{claim: length}For every non negative integer $n$, every $x,y\in A$ with $x\preceq y$ and $d_{Inc(P)}(x,y)= n$, the chains of $(P, \leq)$ included into $[x, y]_{\preceq}$ have at most $n$ elements.
\end{claim}
\noindent{\bf Proof of Claim \ref{claim: length}.} Induction on $n$. The case $n=0$ is obvious since $x=y$. The case $n=1$ is Claim \ref{claim: length1}.
Suppose $n\geq 2$. Then there is some $z\in ]x, y[_{\preceq}$ such that $x$ and $z$ are incomparable with respect to $\leq$ and $d_{Inc(P)}(z,y)= n-1$.\hfill $\Box$

Let $A$ be a connected component of $Inc(P)$. If $1+\omega^*$ is embeddable in $(A, \leq)$ then the interval in the chain $(P, \preceq)$ determined by the images of the extremal elements of $1+\omega*$ contains an infinite chain with respect to $\leq$. According to Lemma \ref{interval}, this interval is included into $A$. But according to Claim \ref{claim: length}, the chains have bounded size. Contradiction. The fact that $\omega+1$ and $1+\omega^*$ do not embed into a poset (semiordered or not) is equivalent to the fact that every chain of $(A, \leq)$ embeds into $\ZZ$ is well known.
\end{proof}

\begin{lemma}\label{lem5:4} Let $G=(X,+,\leq)$ be an ordered group. Then $I(G)(0)$ is a convex normal subgroup and $G/I(G)(0)$ is totally ordered.
\end{lemma}
\begin{proof} According to Lemma \ref{interval}, the poset $(X, \leq)$ is the lexicographical sum of its connected components. The order relation being compatible with the group operation, its complement is compatible, hence the transitive closure of this complement, that is the connectedness relation $\overline \nsim$ is an equivalence relation compatible with the group operation on $G$, hence the equivalence class of $0$, that is $I(G)(0)$, is a normal subgroup and $G/I(G)(0)$ is totally ordered.
\end{proof}

\begin{lemma}\label{lem4:0}Every  subgroup $H$ of an ordered group $G= (X,+,\leq)$ which contains $inc(0)$ contains $I(G)(0)$. \end{lemma}
\begin{proof}We prove that every $x\in I(G)(0)$ belongs to $H$ by induction of $d_{Inc(G)}(0, x)$. The case $n\leq 1$ amounts to the hypothesis. Suppose $n>1$. Then there is some $y$ incomparable to $x$ such that $d_{Inc(G)}(0,y)=n-1$. Via the induction hypothesis, $y\in H$; since $y$ is incomparable to $x$, $-y+x$ is incomparable to $0$, hence belongs to $H$. Consequently $x= y+ (-y +x)\in~H$. \end{proof}

\begin{lemma}\label{lem5:0} Let $G=(X,+,\leq)$ be an ordered group and $H$ be subgroup of $G$. If $H\subseteq inc(0)\cup \{0\}$, then $H$ is an antichain.
\end{lemma}
\begin{proof}Suppose for a contradiction that $H$ is not an antichain and let $a<b$ in $H$. Then $0<b-a\in H$ which is not possible since $H\setminus \{0\}\subseteq inc(0)$.
\end{proof}
\begin{lemma}\label{lem7:0} Let $G=(X,+,\leq)$ be an ordered group and $H$ be a convex subgroup of $G$ distinct from an antichain. Then $H$ is a convex subset of $\leq_{pred}$. In particular $K(G)\subseteq H$.
\end{lemma}
\begin{proof}Indeed, let $x\leq_{pred} z\leq_{pred} y$ and $x,y\in H$. Since $H$ is not an antichain, there exists $h\in H$ such that $0< h$. Hence, $y<y+h$ and since $-h<0$ we have $x-h<x$. By definition of $\leq_{pred}$ we have $x-h<z<y+h$. But, $x-h, y+h\in H$ and $H$ is convex, it follows that $z\in H$. Since $K(G)$ is the equivalence class of $\equiv_{pred}$ and $H$ is a convex subgroup of $(X,+,\leq_{pred})$ we have $K(G)\subseteq H$.
\end{proof}

\begin{lemma}\label{lem3:6} Let $G=(X,\leq,+)$ be an ordered group whose order is an semiorder distinct from an antichain. If $G=I(G)(0)$, then every maximal chain in the positive cone $C^+:=\{x\in X : x\geq 0\}$ has order type $\omega$.
\end{lemma}
\begin{proof}
If $L$ is a maximal chain in $C^{+}$, then it has a minimal element, namely $0$; since $C^{+}\not = \{0\}$,  it has no maximal element. According to Lemma \ref{lem2:6} it embeds in $\ZZ$. The conclusion follows.
\end{proof}

If $G:= (X, +, \preceq)$ is a totally ordered group and $x\in X$ we denote by $\widetilde{\ZZ x}$ the least convex subgroup of $G$ containing $x$. As it is easy to show, $\widetilde {\ZZ x}= \{y\in X: nx\preceq y\preceq mx \; \text {for some} \; n,m\in \ZZ\}$. Alternatively, supposing $0\prec x$, $\ZZ x= \ZZ ]0, x]$ where $]0, x]:= \{y\in X: 0\prec y \preceq x]$.

\begin{lemma}\label{lem4:6}
If  $G:= (X, +, \leq)$ is a non totally ordered threshold group with an attained threshold, say $\alpha$, then $I(G)(0)= \widetilde {\ZZ \alpha}$, the least convex subgroup of $(X, +, \leq_{pred})$ containing $\alpha$.
\end{lemma}

\begin{proof} The set $inc(0)\cup \{0\}$ is the interval $]-\alpha, \alpha[$ of $(X, \leq_{pred})$ hence $\widetilde {\ZZ \alpha}$ contains $inc(0)$. According to Lemma \ref{lem4:0} it contains $I(G)(0)$. Since $inc(0)
\not =\{0\}$ we may pick $b$ with $0<_{pred} b<_{pred} \alpha$. Since $b$ is incomparable to $0$ and $\alpha$ with respect to the order $\leq$, $\alpha\in I(G)(0)$ and more generally $\ZZ \alpha \subseteq I(G)(0)$. Since $I(G)(0)$ is convex in $(X, \leq_{pred})$ it contains $\widetilde {\ZZ \alpha}$.
\end{proof}

\begin{corollary}\label{cor:4:6}
If a threshold group $G$ is not totally ordered and has no infinite antichain then $I(G)(0)$ is the group of integers equipped with a threshold order.
\end{corollary}

\begin{proof}
Assume $G$ is a threshold group (so that $\leq_{pred}$ is a total order) with no infinite antichain. Since $inc(0)$ is the union of (at most) two antichains, it must be finite. Since $inc(0)\cup \{0\}$ is an interval of $\leq_{pred}$ (cf. Item (4) of Subsection \ref{normal}) and this interval is not reduced to $0$, we infer that $0$ has a successor $a$ in $\leq_{pred}$ and $a\in inc(0)$. Let $b$ be the largest element of $inc(0)$ (with respect to $\leq_{pred}$). Then $\alpha= a+b$ is the threshold of $G$. It follows from Lemma \ref{lem4:6} that $I(G)(0)=\ZZ a$.\end{proof}

\begin{lemma}\label{lem2:4}Let $G:=(X,+,\leq)$ be an ordered group whose order is a semiorder. Then $A(G)$ is the largest subgroup $G$ which is an antichain. This is a normal subgroup of $G$ and a convex subset of $(X, \leq)$. \end{lemma}
\begin{proof}

Let $x\in A(G)$ and $\widetilde {\ZZ x}$ be the least convex subset of $(X, \leq_{pred})$ containing $\ZZ x$. Since $inc(0) \cup \{0\}$ is an interval of $\leq_{pred}$ (cf. Item (4) of Subsection \ref{normal}) and $\ZZ x\subseteq inc(0)\cup \{0\}$, it follows that $\widetilde {\ZZ x}\subseteq inc(0)\cup \{0\}$. Trivially, $\widetilde {\ZZ x}$ is a subgroup of $G$.
Now the convex subgroups of a totally quasi-ordered group form a chain with respect to inclusion. Hence, the union of any subfamily is a convex subgroup. In particular, the union of the convex subgroups included into $inc(0)\cup \{0\}$ is a subgroup of $G$.
This subgroup is equal to $A(G)$. By maximality, it is normal.
\end{proof}

\begin{lemma}\label{lem3:4} Let $G:= (X, +, \leq)$ be an ordered group whose order is a semiorder and let $x,y$ be such that $0<_{pred} x\leq_{pred} y$ and $x\not \in A(G)$. If $G=I(G)(0)$, then there exists an integer $n$
such that $y\leq_{pred} n\,x$.
\end{lemma}
\begin{proof}According to Lemma \ref{lem2:4}, $A(G)$ is an interval of $(X, \leq_{pred})$, hence $y\not \in A(G)$. Since $x\not \in A(G)$ there exists some integer $n\in \NN$ such that $0< n x$. We claim that there is some $k\in \NN$ such that $y\leq _{pred} knx$ and we are done. If not, the set $\{knx : k\geq 1\}\cup \{y\}$ is a chain of order type $\omega +1$ in $(X, \leq)$ contradicting Lemma \ref{lem3:6}. Indeed, from $0<nx$ we get that $\{knx : k\geq 1\}$ is a subchain of $(X, \leq)$. Next, from $0<knx$ and $knx\leq_{pred} y$ and $0<knx$ we get $knx <y$, hence $y$ is above all $knx$. \end{proof}

\begin{lemma}\label{lem4:4}Let $G:= (X, +, \leq)$ be a threshold group. Then $I(G)(0)/A(G)$ is an additive subgroup of the reals equipped with a threshold order. If $A(G)\not = \{0\}$ the threshold is attained.
\end{lemma}
\begin{proof} Let $H:= A(G)$. Since $H$ is convex in $\leq_{pred}$, the image $\preceq$ of $\leq_{pred}$ on $I(G)(0)/H$ is a total order. Similarly, since $H$ is a convex in $(X, \leq)$, the image $\leq'$ of $\leq$ on $I(G)(0)/H$ is an order. Let $F$ be the set of positive elements of $G$ with respect to $\leq$. Clearly, the image $F'$ of $F$ is a final segment of $\preceq$. Let $x',y'\in G/H$. We claim that $x'<'y'$ if and only if $y'-x'\in F'$. Suppose that $y'-x'\in F'$. Then, there is some $u\in (y'-x') \cap F$, that is $y'=x'+u+H$, amounting to $x'<y'$. The converse is immediate. From this follows that $G/H$ is a threshold order.
From Lemma \ref{lem3:4} we deduce that $I(G)(0)/H$ is archimedean (with respect to $\preceq$) and from H\"{o}lder's Theorem we deduce that $I(G)(0)/H$ is isomorphic to some subgroup of the additive group of real numbers with the natural order, in particular it is abelian. Suppose that $H\not = \{0\}$. Since $G$ is a threshold order, $(b)$ of Proposition \ref{prop:thresholdgroup} asserts that $G\setminus (-F\cup F)$ is not a union of cosets of a nontrivial convex subgroup $(G, +, \leq_{pred})$. In particular it is not the union of cosets of $H$. Hence some coset meets both $G\setminus (-F\cup F)$ and $F$. This coset is the threshold of $G/H$.
\end{proof}

\subsection{Proof of Theorem \ref{thm:three groups}}\label{pfthmthree groups}
$(a)$. Lemma \ref{lem:interval}.

$(b)$. Lemma \ref{lem5:4} and Lemma \ref{lem4:0}.

$(c)$. Lemma \ref{lem2:4}.

\subsection{Proof of Theorem \ref{thm:3}}

We prove $(i)\Rightarrow (iii) \Rightarrow (ii) \Rightarrow (i)$.

Let $\leq$ the quasi-order on $G$, $\leq_{pred}$ the corresponding quasi-order and $H:= K(G)$.

$(iii) \Rightarrow (ii)$. Since $G/H$ is a threshold group, $I(G/H)(0)$ is a threshold group. Since $I(G/H)(0)= I(G)(0)/H$ the conclusion follows.

$(ii) \Rightarrow (i)$. The order on $G$ is the lexicographical sum of $I(G)(0)$ indexed by $G/I(G)(0)$. Since $I(G)(0)/K(G)$ is a threshold group, the order on $I(G)(0)$ is a semiorder. The order on $G/I(G)(0)$ is total, hence the order on $I(G)(0)\times (G/I(G)(0))$, that is on $G$, is a semiorder

\subsection{Proof of Theorem \ref{thm:4}}

$(a)$. Lemma \ref{lem4:4}.

$(b)$. Apply Lemma \ref{lem:interval} to $I(G)(0)$ and $H=K(G)$.

$(c)$. Let $C:= \{x\in X: 0<x\}$. According to Lemma \ref{lem4:4}, some coset of $H:=A(G)$  meets both $C$ and $G\setminus C$. Let $F'$ be the intersection of this coset with $C$. Since $A(G)$ is a direct factor of $G$ we may suppose that the group $G$ is the direct product $A(G)\times (G/A(G))$, hence $F'= F\times \{\alpha\}$ where $\alpha\in G/A(G)$ and $F$ is a final segment of $A(G)$ equipped with the order induced by $\leq_{pred}$. It is easy to check that the order $\leq_{F, \alpha}$ coincide with $\leq$.

\subsection{Proof of Theorem \ref{thm:6}} (1) Follows from Lemma \ref{lem3:6}. We now prove (2). If $G$ is isomorphic to the group of integers equipped with a threshold order, then clearly $K(G)=\{0\}$ and $G$ has no infinite antichain. Now suppose $K(G)=\{0\}$ and $G$ has no infinite antichain. It follows from $K(G)=\{0\}$ that $\leq_{pred}$ is a total order. Then we conclude from Lemma \ref{lem4:6} that $I(G)(0)$ is the least convex subgroup of $(X,+,\leq_{pred})$. Since $G$ has no infinite antichain and $inc(0)\cup \{0\}$ is bipartite (see Theorem \ref{claim:interval} ) we infer that $inc(0)\cup \{0\}$ is finite. Since the set $inc(0)\cup \{0\}$ is the interval $]-\alpha, \alpha[$ in $\leq_{pred}$ it follows that $]-\alpha, \alpha[$ is finite. It follows by translation that $0$ has a successor in $\leq_{pred}$. Say $a$ is the successor of $0$ in $\leq_{pred}$. It follows easily that $I(G)(0)$ is $\ZZ a$ and hence $G=I(G)(0)$ is isomorphic to the group of integers equipped with a threshold order. The proof of Theorem \ref{thm:6} is now complete.

\subsection{Proof of Theorem \ref{convexsubgroup}} Let $G:=(X,+,\leq)$ be an ordered group such that $\leq$ is a semiorder and $H$ be a convex subgroup of $G$. If $H$ is  an antichain then by (c) of Theorem \ref{thm:three groups}, $H\subseteq A(G)$ and we are done. Suppose that $H$ is not an antichain. From Lemma \ref{lem7:0} we deduce that $H$ is an interval of $\leq_{pred}$ containing $K(G)$. Now,  suppose for a contradiction that $inc(0)\nsubseteq H$ and let $g\in
inc(0)\cup \{0\}\setminus H$. Then $g$ is either above all elements of $H$ or is below all elements of $H$ in $\leq_{pred}$. Say the former holds. Since $H$ is not an antichain we infer that there exists $h\in H$ such that $0<h$. By definition of $\leq_{pred}$ we have that $0<g$ which is impossible. Hence, $inc(0)\subseteq H$. According to Lemma \ref{lem4:0}, $I(G)(0)\subseteq H$. Since $I(G)(0)$ is normal in $G$, it is normal in $H$, hence $H$ is the union of the cosets of $I(G)(0)$. These cosets are totally ordered by Lemma \ref{lem5:4}. If $H$ is normal, it follows that the cosets of $H$ form a chain of intervals and $G$ is the lexicographical sum of the cosets of $H$.
\hfill $\Box$

\section{Clifford's example and a proof of Theorem \ref{fingen}}\label{subsection-clifford} In this section we exhibit an example, due to Clifford, of an ordered group with the property that the set of positive elements contains no proper normal final segment.

Following Clifford \cite{clifford51}, let $G$ be the group generated by a set of symbols $g(\alpha)$, one for each rational number $\alpha$, subject to the generating relations
\begin{equation}\label{exclif1}
g(\alpha)+ g(\beta) = g(\frac{\alpha+\beta}{2})+ g(\alpha)\;\;\;\;\;\; \mbox{ if } \alpha > \beta.
\end{equation}
By repeated application of (\ref{exclif1}) it is clear that every element $a$ of $G$ can be brought to "normal form":
\begin{equation}\label{exclif2}
a=m_1\, g(\alpha_1)+ m_2\, g(\alpha_2)+\cdots +m_s\, g(\alpha_s).
\end{equation}
where the $m_i$ are nonzero integers, and the $\alpha_i$ are rational numbers satisfying $\alpha_1<\alpha_2< \cdots <\alpha_s.$ It can be shown that the normal form (\ref{exclif2}) is unique. In particular, $g(\alpha)\neq g(\beta)$ for $\alpha \neq \beta$. We order $G$ by declaring $a\succ 0$ if $m_s>0$. Note that $0<g(\alpha)<g(\beta)$ whenever $\alpha < \beta$.

\begin{proposition} \label{clifford}
A nontrivial final segment of $\{g\in G: g\succ 0 \}$ cannot be normal.
\end{proposition}
\begin{proof}
\textbf{Claim 1:} A normal subset $N$ of $G$ either contains all generators of $G$ or contains none.\\
Indeed, let $\beta$ be a rational number such that $g(\beta) \in N$. Let $r$ be a rational number such that $r>\beta$. Then $2r-\beta>\beta$, hence by applying (\ref{exclif1}) we get $g(2r-\beta)+ g(\beta) = g(r)+ g(2r-\beta)$, that is, $g(2r-\beta)+ g(\beta)- g(2r-\beta) = g(r)$. From the normality of $N$ we deduce that $g(r)\in N$. We now consider the case $\beta>r$. Then $2\beta -r>r$, hence by applying (\ref{exclif1}) we get $g(2\beta-r)+ g(r) = g(\beta)+ g(2\beta-r)$, that is $g(r)=-g(2r-\beta)+ g(\beta)+ g(2r-\beta)$. From the normality of $N$ we deduce that $g(r)\in N$. Hence, all generators of $G$ are in $N$.\\
\textbf{Claim 2:} A normal subset $N$ which is a final segment of $G^+$ and that contains all generators of $G$ must contain $G_*^+$.\\
Indeed, let $a\in G^+_*$. Write $a$ in normal form as in (\ref{exclif2}). By definition, $m_s >0$ from which it follows that $g(\alpha_1) \prec a$ (this is obvious if $s=1$; if $s\not =1$ observe that  $-g(\alpha_1) + a = (m_1 - 1)\, g(\alpha_1)+ m_2\, g(\alpha_2)+\cdots +m_s\, g(\alpha_s)$ is under normal form with $m_s >0$ hence is positive). Thus $g(\alpha_1)\prec a$ from which $a\in N$ follows. This proves that $G^+_*\subseteq N$.\\
\textbf{Claim 3:} A normal subset $N$
which is a final segment of $G^+$ and that contains none of the generators of $G$ must be empty.\\
Suppose that none of the generators of $G$ belong to $N$. Let $I:=G\setminus (-N\cup N)$. Since all generators belong to $G^{+}$ they belong to $I$. We claim that for every $a\in G$ there is some rational $\beta$ such that $a\prec g(\beta)$. For that, we prove that $n\, g(\alpha)\prec g(\beta)$ for every integer $n$ and rational numbers $\alpha$ and $\beta$ such that $\alpha<\beta$. Let $a= -ng(\alpha)+g(\beta)$ and notice that since $\alpha<\beta$ the element $a$ is in normal form. It follows then that $0\prec a$, that is, $ng(\alpha)\prec g(\beta)$. Now let $a\in G$ written in normal form, say $a=m_1\, g(\alpha_1)+ m_2\, g(\alpha_2)+\cdots +m_s\, g(\alpha_s)$. We have that $m_ig(\alpha_i)\prec g(\alpha_s)$ for all $1\leq i<s$ and therefore $a\preceq (s-1+m_s)\, g(\alpha_s)\prec g(\beta)$ for a rational $\beta$ such that $\alpha_s<\beta$. From this follows that $a\in I$. This proves $I=G$ and hence $N=\varnothing$ as required.
\end{proof}

\subsection{Proof of Proposition \ref{fingen}}\label{section:fingen}
We will use abelian quotients.
\begin{theorem}\label{finalsegment:abelianquotient}
 Let $G$ be a totally ordered group and $H$ be a nontrivial normal convex subgroup. If $G/H$ is abelian, then $G$ has a normal final segment $F$ such that $I:=G\setminus (-F\cup F)$ is not a subgroup of $G$.
\end{theorem}
\begin{proof}Let $\alpha \in G$ be such that $0<\alpha$ and $\alpha \not \in H$ and set
\[F:=\{g \in G : g+H> \alpha +H\}.\]

\begin{itemize}
\item $F\neq \varnothing$. Indeed, since $I$ is nontrivial $G/H$, which is totally ordered, must be infinite.
\item $\alpha \not \in F$.
\item Every element of $F$ is positive. Let $g\leq 0$. Then $g+H\leq 0+H<\alpha +H$ and hence $g\not \in F$.
 \item $F$ is a final segment of $G$. Let $f,g\in G$ be such that $f\in F$ and $f\leq g$. Then $\alpha +H < f+H \leq g+H$. Hence, $g\in F$.
 \item $F$ is normal. Let $f\in F$ and $g\in G$. Then
\begin{eqnarray*}
 g+f-g+H &=& (g+H)+(f+H)-(g+ H) \\
  &=& (g+H)-(g+ H)+(f+H) \mbox{ (this is because $G/H$ is abelian)}\\
  &=& f+H> \alpha +H.
\end{eqnarray*}

\item $I$ is not a subgroup of $G$. Note that since $\alpha \not \in F$ and $0<\alpha$ we infer that $\alpha \in I$ (in particular $I\neq H$). Since $2\alpha - \alpha =\alpha \not \in H$ we infer that $2\alpha +H\neq \alpha+H$. From $\alpha <2\alpha$ we infer that $\alpha +H <2\alpha +H$ leading to $2\alpha \in F$. Hence, $2\alpha \not \in I$ proving that $I$ is not a subgroup of $G$.
\end{itemize}
\end{proof}

We now consider finitely generated groups.

\begin{theorem}\label{abelianquotient:nonarchimedean} Let $G$ be a finitely generated nonarchimedean totally ordered group. Then $G$ has a nontrivial normal convex subgroup $H$ such that $G/H$ is abelian.
\end{theorem}
\begin{proof}
Let $g_1,\cdots,g_n$ be generators of $G$. Note that  $n\geq 2$. We may assume without loss of generality that $0<g_1<\cdots<g_n$.  We prove that $C_{g_n}$, the largest convex subgroup not containing $C_{g_n}$ (whose existence is ensured by Lemma \ref{largestconvex}),  satisfies the conditions of the theorem.\\
$(i)$ Clearly $C_{g_n}\neq G$. Moreover, if $H$ is a convex subgroup containing properly $C_{g_n}$, then, by definition, $g_n \in H$. By convexity, all the $g_i$'s belong to $H$, hence $H=G$. Thus, $C_{g_n}$ is the largest proper convex subgroup of $G$.\\
$(ii)$ Since $G$ is nonarchimedean it has a nontrivial convex subgroup. This subgroup cannot contain $g_n$ and hence $C_{g_n}\neq \{0\}$.\\
(iii) We now prove that $C_{g_n}$ is a normal subgroup of $G$. Indeed, if $g\in G$, then $g+C_{g_n}-g$ is a convex subgroup (this is because the map $\varphi_{g}(x)=g+x-g$ is an ordered group automorphism) and hence either $C_{g_n}\subseteq g+C_{g_n}-g$ or $g+C_{g_n}-g\subseteq C_{g_n}$. Note that $g_n\not \in g+C_{g_n}-g$ because otherwise $G=g+C_{g_n}-g$ and hence $G=C_{g_n}$ contradicting the fact that $C_{g_n}$ is a proper subgroup. Hence, $g+C_{g_n}-g\subseteq C_{g_n}$ for all $g\in G$ proving that $C_{g_n}$ is a normal subgroup of $G$.\\
$(iv)$ Since $C_{g_n}$ is the maximal proper convex subgroup of $G$ we infer that $G/C_{g_n}$ has no nontrivial convex subgroups. Therefore $G/C_{g_n}$ is archimedean and hence is isomorphic to a subgroup of $\RR$. In particular it is abelian.
\end{proof}
We now proceed to the proof of Proposition \ref{fingen}. Let $G$ be a totally ordered group. If $G$ is archimedian, then it follows from H\"older's Theorem that $G$ is abelian. From Corollary \ref{cor:threshold-abelian} we deduce that $G$ has a nonempty final segment $F$ verifying $(b)$ of Proposition \ref{cor:threshold-abelian} hence $H$ is not a subgroup. On the other hand, if $G$ is nonarchimedian and finitely generated, then the required conclusion follows from Theorems \ref{finalsegment:abelianquotient} and \ref{abelianquotient:nonarchimedean}.

\section{Dimension} \label{subsection:dimension}Let $G$ be a threshold group, with attained threshold $\alpha$. Our aim is to prove in an effective way that the dimension of the order is at most $3$.

\begin{proposition}\label{thresholddimension}A threshold  group $G:= (X, +, \leq_{\alpha})$ with attained threshold $\alpha$ has dimension at most $3$.
\end{proposition}
\begin{proof}
 Since $G$ is a threshold group, it follows from $(b)$ of Theorem \ref{thm:three groups} that the order on $G$ is
 isomorphic to the lexicographical sum of the order of $I(G)(0)$ indexed by the chain $G/I(G)(0)$. Since the dimension of a linear sum is at most the dimension of each component of the sum, it suffices to prove that the dimension of $I(G)(0)$ is at most three. If $I(G)(0)= \{0\}$, $G$ is totally ordered and the result is obvious. We suppose that $I(G)(0)\not = \{0\}$ and we prove that the order $\leq_{\alpha}$ is the intersection of the three total orders $\leq^{i}$, for $i=1,2,3$. According to Lemma \ref{lem4:6}, $I(G)(0)= \widetilde {\ZZ \alpha}$. We define the orders $\leq^{i}$, for $i=1,2,3$ on intervals of $I(G)(0)$ of the form $]n\alpha, (n+1)\alpha]$,
 $]2n \alpha, (2n+2) \alpha]$ and $](2n+1)\alpha, (2n+3)\alpha]$ respectively. Then, we extend these orders to $I(G)(0)$ by putting these intervals one after the other according to the order of the integers $n\in \ZZ$. The order $\leq^{1}$ reverses the order $\leq_{pred}$ on each $]n\alpha, (n+1)\alpha]$. Being the lexicographical sum of the $]n\alpha, (n+1)\alpha]$ indexed by the chain $\ZZ$, $I(G)(0)$ is totally ordered and this order extends the order $\leq$. The order $\leq^{2}$ on $]2n \alpha, (2n+2) \alpha]$ and the order $\leq^{3}$ on $](2n+1)\alpha, (2n+3)\alpha]$ are defined in the same way.
 To define these two orders  at once, let $A_u:= ]u, u+2\alpha]$ be an interval of length $2\alpha$ of $(X, \leq_{pred})$. The relation $\preceq^{u} $ we define reverses all pairs $(x, y)$ with $x\in ]u, u+\alpha]$ and $y\in ]u+\alpha, u+2\alpha]$, and $x\not \leq y$ (that is $y\preceq^{u} x$) and keep unchanged the others pairs $(x,y)$ such that $x\leq_{pred}y$. This relation is reflexive, antisymmetric, total and extends the order $\leq$. We claim that it is transitive. Since it is total, it suffices to show that there is no $3$-element cycle. If there is such a cycle $\{x,y,z\}$, two of its elements, 	at least, say $x,y$ with $x\preceq y$,  are in $]u, u+\alpha]$ or in $]u+\alpha, u+2\alpha]$. Suppose that they are in $]u, u+\alpha]$. Since $\preceq^{u}$ coincide with $\leq_{pred}$ on $]u, u+\alpha]$ we have $x\leq_{pred}y$ and necessarily the third element $z$ is in $]u+\alpha, u+2\alpha]$. Since $\{x,y, z\}$ forms a cycle we must have $y<z$ and $x\not \leq z$. But $x\leq_{pred}y$ amounts to $x\leq_{succ} y$; with $y<z$, this implies $x<z$, a contradiction. The case where two elements of the cycle are in $]u+\alpha, u+2\alpha]$ is similar. Now, we take $\leq^{2}$ equal to $\preceq^{u}$ for $u= 2n\alpha$ and $\leq^{3}$ equal to $\preceq^{u}$ for $u= (2n+1)\alpha$. The resulting orders on $I(G)(0)$ extend the order $\leq$. To conclude, we show that their intersection is $\leq$. Let $x, y\in X$ such that $x\leq ^{i} y$ for $i=1, 2, 3$; we prove that $x\leq y$. Supposing $x\not =y$ this amounts to prove that $\alpha\leq_{pred} y-x$. First, observe that $x$ and $y$ cannot belong to the same interval $]n\alpha, (n+1) \alpha]$, indeed on such an interval $\leq^{1}$ and $\leq^{2}$ are reverse of each other. If these two intervals are not consecutive, we have $x\leq_{pred} y$ and $\alpha\leq_{pred} y-x$ hence $x\leq y$ as required. If the two intervals are consecutive, we have $x\in ]n\alpha, (n+1) \alpha]$ and $y\in ](n+1)\alpha, (n+2) \alpha]$. Indeed, with respect to $\leq^{2}$ if $n$ is odd, and with respect to $\leq^{3}$ if $n$ is even, all elements of the interval $]n\alpha, (n+1) \alpha]$ are before those of $](n+1)\alpha, (n+2) \alpha]$. If $n$ is even, the fact that $x\leq^{2}y $ implies $x\leq y$ whereas  if $n$ is odd, $x\leq^{3}y $ implies $x\leq y$. \end{proof}

\begin{remark} None of the linear orders $\leq^{i}$, for $i=1,2,3$, is compatible with the group operation (still $\leq^1$ is preserved by the translation $t_{\alpha}$, whereas $\leq^{2}$ and $\leq^{3}$ are preserved by $t_{2\alpha}$. There are very rare cases for which the order of an ordered group is the intersection of compatible linear orders. In fact, in many cases, there is just one compatible linear order extending the order. For an example, this is the case for the threshold order $\leq_{1}$ on $\RR$.
\end{remark}

\section{Posets embeddable into ordered groups}\label{section:comparison}
We introduce the following notion. Let $P, Q$ be two posets, we set $P \preceq Q$ if $P$ can be embedded in every ordered group $G$ in which $Q$ can be embedded. This defines a quasi-order  on the class of posets. This quasi-order extends the embeddability relation, that is if $P$ embeds in $Q$, then $P\preceq Q$. The equivalence relation corresponding to this quasi-order is defined by $P\simeq Q$ if $P\preceq Q$ and $Q\preceq P$. We set $P\prec Q$ if $P\preceq Q$ and $Q\not \preceq P$. Observe that if $P$ is a 1-element chain, the equivalence class of $P$ consists of the 1-element chains. Hence, if we consider posets up to isomorphism, this class has just one element. Similarly:

\begin{enumerate}
  \item[($\alpha$)] \emph{If $P$ is a $n$-element antichain, the equivalence class consists of "the" $n$-element antichain}.\\
Indeed, we claim that if $Q\prec P$ then $Q$ is an antichain of size at most $n-1$. For that, let $G$ be the direct product $G: = \ZZ\times (\ZZ/n\ZZ)$. Select as  positive cone the set $C_1:= \{(m,0): m\geq 0\}$. The resulting ordered group is a direct sum of $n$ copies of $\ZZ$. Since this poset embeds $P$, it embeds $Q$, hence $Q$ is a direct sum of chains. Select as positive cone $C_2:= C_1\cup \{(m,i): 1\leq m, i<n\}$. The order is a lexicographical sum over $\ZZ$ of $n$-element antichains. Again, this poset embeds $P$ hence it embeds $Q$, hence $Q$ is a lexicographical sum of antichains. Necessarily, $Q$ is an antichain and since it cannot be isomorphic to $P$ it has less than $n$ elements proving our claim.

\item[($\beta$)]\emph{If $P$ is the $2$-element chain, then its equivalence class is made of all finite chains of cardinality at least $2$, of the chain $\omega$ of nonnegative integers, the chain $\omega^*$ of negative integers and the chain $\ZZ$ of all integers}.\\
Indeed, if an ordered group $G$ contains a $2$-element chain, it contains some element $x$ such that $0<x$. The subgroup generated by $x$ is $x\ZZ$; it is ordered as $\ZZ$ hence $G$ contains a chain of type $\ZZ$.

\item[($\gamma$)]\emph{If $P$ is a direct sum of two finite chains and $Q\preceq P$ then $Q$ is also a direct sum of two chains.}\\
Indeed, let $Q \preceq P$; The group $G: = \ZZ\times (\ZZ/2\ZZ)$ with positive cone $C_1:= \{(m, 0): m\geq 0\}$ being the direct sum of two infinite chains, it embeds $P$. Since $Q\preceq P$, it embeds $Q$, hence $Q$ is a direct sum of two chains.

\item[($\delta$)]\emph{If $P$ is $1 \oplus 2$, its equivalence class reduces to $1\oplus 2$.}\\
Indeed, from (c) $Q$ is a direct sum of (at most) two chains. The ordered group $G:= (\ZZ, \leq)$, with $C_3:=\{n: n \geq 3\}$ as positive cone, embeds $P$, hence embeds $Q$. It follows that either $Q$ is a chain or is isomorphic to $P$. This proves our claim.
\end{enumerate}
We prove:

\begin{proposition}\label{equivalence class}
\begin{enumerate}
\item For every positive integer $n\geq 2$, the set $\mathcal S_n$ of posets of the form $(q+1)\oplus p$ such that $p, q\geq 1$ and $p+q=n$ form an equivalence class of $\simeq$.
\item Let $P\in \mathcal S_n$. Then for every poset $Q$:
\begin{enumerate}
\item if $Q\prec P$, then either $Q \simeq 1$ or $Q \simeq 2$ or $Q$ belongs to some $\mathcal S_{n'}$ for some $n'<n$;
\item if $n=3$ and $P \preceq Q$, then either $2\oplus 2$ or $1\oplus 3$ embeds in $Q$.
\end{enumerate}
\end{enumerate}
\end{proposition}
\begin{proof}
(1). As we have seen with Theorem \ref{2+2,3+1}, all posets of the form $(q+1) \oplus p$ for positive $p$ and $q$ such that $p+q=n$ are equivalent. To see that the equivalence class contains nothing else, let $Q\preceq P$ with $P\in \mathcal{S}_n$. According to ($\gamma$), $Q$ is a direct sum of two chains. Let $G$ be the group $\ZZ$ with positive cone $2\NN\cup (n+1)\NN$ if $n$ is even and equal to $\ZZ\times \ZZ/2\ZZ$ with strict positive cone generated by $(1,0)$ and $(p+1,1)$ if $n$ is odd, $n=2p+1$. This group contains $1 \oplus n$ but not $1\oplus (n+1)$. Hence, according to Theorem \ref{2+2,3+1}, it contains $P$ and hence $Q$. If $Q$ is a chain then $Q$ is equivalent to the $1$- element  or the $2$-element chain,  and thus $Q\prec P$. If $Q$ is not a chain, then via Theorem \ref{2+2,3+1}, the cardinality of $Q$ is at most $n+1$. If $\vert Q\vert <\vert P\vert $, then $Q\prec P$, whereas, if $\vert Q\vert =\vert P\vert $, $Q$ is of the form $(q+1)\oplus p$ with $p+q=n$ and $Q\preceq P$.

Let us prove $(2)$.

The argument above proves that $(a)$ holds.

$(b)$ From $P \preceq Q$ we have that every group embedding $Q$ embeds $P$. According to Theorem \ref{thm:1}, if neither $2 \oplus 2$ nor $3 \oplus 1 $ embeds into $Q$ then there is a group with the same property extending $Q$, but this contradict the fact that $P \preceq Q$. Hence, either $2 \oplus 2$ or $3 \oplus 1 $ embeds into $Q$.
\end{proof}

\begin{problems} Let $n\geq 4$ and let $P\in \mathcal S_n$. Is it true that if $P \preceq Q$ then $P'$ embeds in $Q$ for some $P'\in \mathcal S_n$? Is it true that every poset embedding no member of  $\mathcal S_n$ can be embedded into an ordered group with the same property?
\end{problems}


\begin{thebibliography}{99}
\bibitem{aleskerov-al} F. Aleskerov, D. Bouyssou,  B. Monjardet, \emph{Utility maximization, choice and preference}.
Second edition. Springer, Berlin, 2007. xii+283 pp.

\bibitem{bogart} K. P. Bogart, \emph{An obvious proof of Fishburn's interval order theorem.} Discrete Mathematics \textbf{118} (1993), 239--242.

\bibitem{brt} K. P. Bogart, I. Rabinovich and W. T. Trotter Jr., \emph{A bound on the dimension of interval orders.} Journal of Combinatorial Theory, Series A \textbf{21} (1976), 319--328.

\bibitem{bgr}
S. J. Brams, W.V. Gehrlein and F.S. Roberts, eds, \emph{The mathematics of preference, choice and order.} Essays in honor of Peter C. Fishburn. Studies in Choice and Welfare. Springer-Verlag, Berlin, 2009. xviii+420 pp. ISBN: 978-3-540-79127-0

\bibitem{bridges}D.S. Bridges, G.B. Mehta, \emph{Representations of preferences orderings}.
Lecture Notes in Economics and Mathematical Systems, 422. Springer-Verlag, Berlin, 1995. x+165 pp.

\bibitem{candeal1} J.C. Candeal, A. Estevan, J. Guti\'errez-Garc\'{i}a, E. Indur\'ain,
\emph{Semiorders with separability properties.}
J. Math. Psych. 56 (2012), no. 6, 444-451.

\bibitem {candeal} J. C. Candeal and E. Indur\'ain, \emph{Semiorders and thresholds of utility discrimination: solving the Scott-Suppes representability problem}. J. Math. Psych. \textbf{54} (2010), 485--490.


\bibitem{CannonFloydParry} J. W. Cannon, W. J. Floyd, and W. R. Parry, \emph{Introductory notes on Richard Thompson's groups.} Enseign. Math. \textbf{42} (1996), 215--256.

\bibitem{clifford51} A. H. Clifford, {\em A noncommutative ordinally simple linearly ordered group}, Proc. Amer. Math. Soc. \textbf{2} (1951), 902--903.



 \bibitem{doignon-al1}J.P. Doignon, A. Ducamp, J.C. Falmagne, \emph{On realizable biorders and the biorder dimension of a relation.} J. Math. Psych. 28 (1984), no. 1, 73--109.

\bibitem{doignon-al2} J.P. Doignon, B. Monjardet, M. Roubens,  Ph. Vincke,
\emph{Biorder families, valued relations, and preference modelling}.
J. Math. Psych. 30 (1986), no. 4, 435--480.

\bibitem{dm} B. Dushnik and E. W. Miller, \emph{Partially ordered sets}, Amer. J. Math. \textbf{63} (1941), 600--610.

\bibitem{ehrenfeucht-al} A. Ehrenfeucht, T. Harju, G. Rozenberg,
 \emph{The theory of $2$-structures. A framework for decomposition and transformation of graphs}. World Scientific Publishing Co., Inc., River Edge, NJ, 1999. xvi+290 pp.


\bibitem{estevan-al} A. Estevan,  M. Schellekens, O. Valero, \emph{Approximating SP-orders through total preorders: incomparability and transitivity through permutations}. Quaest. Math. {\bf 40} (2017), no. 3, 413--433.



\bibitem{fi} P. C. Fishburn, \emph{Intransitive indifference with unequal indifference intervals.} J. Math. Psych. \textbf{7} (1970) 144--149.

\bibitem{fi-book} P. C. Fishburn. \emph{Interval orders and interval graphs.} John Willey \& Sons, 1985.


\bibitem{fraisse} R. Fra\"{\i}ss\'e. \emph{On a decomposition of relations which generalizes the sum of ordering
relations}, Bull. Amer. Math. Soc. 59 (1953), 389.

\bibitem{gallai} T.Gallai, \emph{Transitiv orientierbare Graphen}, Acta Math. Acad. Sci. Hungar 18 (1967),
25-66.

\bibitem{fuchsbook} L. Fuchs, \emph{Partially ordered algebraic systems}, Pergamon Press, 1963.



\bibitem{fhrt}Z. F\"{u}redi, P. Hajnal, V. R\"{o}dl, W.T. Trotter, \emph{Interval orders and shift graphs}, in: A. Hajnal, V.T. S\`{o}s,(Eds.), Sets, Graphs and Numbers, Colloq. Math. Soc. Janos Bolyai, Vol. \textbf{60} (1991), 297--313.




\bibitem{glassbook} A. Glass, \emph{Partially ordered groups}, World Scientific, 1999.


\bibitem{holder} O. H\"{o}lder, {\em Die Axiome der Quantitat und die Lehre vom Mass}, Berichte uber die Verhandlungen der Koeniglich Sachsischen Gesellschaft der Wissenschaften zu Leipzig, Mathematisch-Physikaliche Klasse \textbf{53} (1901), 1--64. (Part 1 translated by J. Michell and C. Ernst (September 1996). {\em The axioms of quantity and the theory of measurement}. J. Math. Psych 40 (3): 235--252.)

\bibitem{kopytov} V. M. Kopytov and N. Ya. Medvedev, \emph{The Theory of Lattice-Ordered Groups}, Mathematics and Its Applications Volume 307, Springer, Dordrecht. ISBN 978-90-481-4474-7.


\bibitem{levi1} F. W. Levi, {\em Arithmetische Gesetze im Gebiete diskreter Gruppen}, Rend. Circ. Mat. Palermo \textbf{35} (1913), 225--236.
%

\bibitem{luce56} R. D. Luce, \emph{Semiorders and a theory of utility discrimination.} Econometrica \textbf{24} (1956), 178--191.

\bibitem{luce87} R. D. Luce, \emph{Measurement Structures with Archimedean Ordered Translation Groups.} Order \textbf{4} (1987), 165--189.

\bibitem{magnus}
W. Magnus, A. Karrass and D. Solitar, \emph{Combinatorial group theory. Presentation
of groups in terms of generators and relations}. 2nd revised edition, Dover Publications, New York, 1976.

\bibitem{manders} K. L. Manders, \emph{On JND Representations of Semiorders.} J. Math. Psych. \textbf{24} (1981), 224--248.



\bibitem{mirsky} L. Mirsky, \emph{A dual of Dilworth's decomposition theorem}, American Mathematical Monthly 78 (8): 876-877.

\bibitem{NavasRivas} A. Navas and C. Rivas, \emph{Describing all bi-orderings on Thompson's group F.} Groups Geom. Dyn. \textbf{4} (2010), 163--177.
\bibitem{pirlotvincke} M. Pirlot and P. Vincke, \emph{Semiorders: Properties, Representations, Applications.} Volume 36 of Theory and Decision Library Series B, Springer Science \& Business Media, 1997.

\bibitem{Pouzet-Thiery2}M. Pouzet, N. Thi\'ery, \emph{Some relational structures with polynomial growth and their associated algebras I. Quasi-polynomiality  of the profile}. The Electronic J. of Combinatorics, 20(2) (2013), 35pp.
\bibitem{robbiano}L.~Robbiano, \emph{Term orderings on the polynomial rings}, EUROCAL '85, Vol. 2 (Linz, 1985), 513--517, Lecture Notes in Comput. Sci., 204, Springer, Berlin, 1985.

\bibitem{rabinovitch} I. Rabinovitch, \emph{The dimension of semiorders}, J. Comb. Theory Ser A \textbf{25} (1978), 50--61.


\bibitem{scottsuppes} D. Scott and P. Suppes, \emph{Foundational aspects of theories of measurement}. The Journal of Symbolic Logic \textbf{23} (1958), 113--128.

\bibitem{shoenfield} J.R. Shoenfield,  \emph{Mathematical logic}. Reprint of the 1973 second printing. Association for Symbolic Logic, Urbana, IL; A K Peters, Ltd., Natick, MA, 2001. viii+344


\bibitem{szp} E. Szpilrajn, {\em Sur l'extension de l'ordre partiel,} Fund. Math., {\bf 16} (1930), 386-389.

\bibitem{thompson}
R. Thompson. Handwritten notes. 1965.

\bibitem{trotterbook} W.T. Trotter, \emph{Combinatorics and partially ordered sets. Dimension theory}. Johns Hopkins Series in the Mathematical Sciences. Johns Hopkins University Press, Baltimore, MD, 1992. xvi+307 pp.

\bibitem{wiener} N. Wiener, \emph{A Contribution to the Theory of Relative Position.} Proc. Cambridge Philos. Soc. \textbf{17} (1914), 441--449.

\end{thebibliography}
\end{document}